\documentclass[a4paper,
11pt,
pdftex,
normalheadings,
headsepline,
footsepline,
onecolumn,
headinclude,
footinclude,
DIV14,
abstracton]
{scrartcl}

\usepackage
{
	graphicx,
	amssymb,
	amsmath,
	amsthm,
	xcolor,
	dsfont,
	algpseudocode,
	authblk,
}

\usepackage[ngerman, english]{babel}

\usepackage[bf]{caption}
\usepackage[colorlinks,
pdffitwindow=false,
plainpages=false,
pdfpagelabels=true,
pdfpagemode=UseOutlines,
pdfpagelayout=SinglePage,
bookmarks=false,
colorlinks=true,
hyperfootnotes=false,
linkcolor=blue,
citecolor=green!50!black]
{hyperref}

\usepackage{enumerate}
\usepackage{algorithm}

\usepackage{tabularx}

\newcommand{\R}{\mathbb{R}}

\newtheorem{theorem}{Theorem}[section]

\newtheorem{lemma}[theorem]{Lemma}

\newtheorem{definition}[theorem]{Definition}
\newtheorem{remark}[theorem]{Remark}
\newtheorem{example}[theorem]{Example}

\let\OLDthebibliography\thebibliography
\renewcommand\thebibliography[1]{
	\OLDthebibliography{#1}
	\setlength{\parskip}{0pt}
	\setlength{\itemsep}{0pt plus 0.3ex}
}

\begin{document}

\title{A Descent Method for Equality and Inequality Constrained Multiobjective Optimization Problems}
\author[1]{Bennet Gebken}
\author[1]{Sebastian Peitz}
\author[1]{Michael Dellnitz}
\affil[1]{\normalsize Department of Mathematics, Paderborn University, Germany}

\maketitle

\begin{abstract}
	In this article we propose a descent method for equality and inequality constrained multiobjective optimization problems (MOPs) which generalizes the steepest descent method for unconstrained MOPs by Fliege and Svaiter to constrained problems by using two active set strategies. Under some regularity assumptions on the problem, we show that accumulation points of our descent method satisfy a necessary condition for local Pareto optimality. Finally, we show the typical behavior of our method in a numerical example.
\end{abstract}

\section{Introduction}
In many problems we face in reality, there are multiple objectives that have to be optimized at the same time. In production for example, one often wants to minimize the cost of a product but also maximize its quality. When the objectives we want to optimize are conflicting (like in the above example), classical scalar-valued optimization methods are not suited. Since there is no point that is optimal for all objectives at the same time, one has to search for \emph{optimal compromises}. This is the motivation for \emph{multiobjective optimization}.

A general \emph{multiobjective optimization problem} (MOP) consists of $m$ objective functions $F_i: \mathcal{N} \rightarrow \R$, where $i \in \{1,...,m\}$ and $\mathcal{N} \subseteq \R^n$. The set of optimal compromises is called the \emph{Pareto set}, containing all \emph{Pareto optimal points}. A point $x \in \mathcal{N}$ is a Pareto optimal point if there exists no feasible point that is at least as good as $x$ in all objective functions and strictly better than $x$ in at least one objective function. The goal in multiobjective optimization is to find the Pareto set of a given MOP.

For unconstrained MOPs (i.e.~$\mathcal{N} = \R^n$), there are multiple ways to compute the Pareto set. A popular approach is to scalarize the MOP -- e.g.~by weighting and summarizing the objective functions (see e.g.~\cite{M1998,E2005}) -- and then solve a sequence of scalar-valued problems. Another widely used approach is based on heuristic optimization and results in evolutionary methods (\cite{D2001,SMDT03,CLV2007}). In the case where the Pareto set is a connected manifold, it can be computed using continuation methods \cite{SDD2005}. Some methods for scalar-valued problems can be generalized to MOPs. Examples are the steepest descent method \cite{FS2000}, the Newton method \cite{FDS2009} and the trust-region method \cite{CLM2016}. Finally, set-oriented methods can be applied to compute a covering of the global Pareto set \cite{DSH2005,SWO+13}.

There also exist gradient-based methods which can handle both unconstrained MOPs as well as certain classes of constraints. If the MOP is constrained to a closed and convex set, it is possible to use a projected gradient method \cite{DI2004}. For more general inequality constraints, it is possible to use the steepest descent method described in \cite[Section 8]{FS2000}. For equality constrained MOPs where the feasible set is a (Riemannian) manifold, it is possible to use the steepest descent method described in \cite{BFO2012}. Until recently, the consideration of MOPs with equality and inequality constraints was limited to heuristic methods \cite{MXXWM2014} and scalarization methods (see e.g.~\cite{E2008}). In 2016, Fliege and Vaz proposed a method to compute the whole Pareto set of equality and inequality constrained MOPs that is based on SQP-techniques \cite{FV2016}. Their method operates on a finite set of points in the search space and has two stages: In the first stage, the set of points is initialized and iteratively enriched by nondominated points and in the second stage, the set of points is driven to the actual Pareto set.

The goal of this article is to extend the steepest descent method in \cite{FS2000} to the constrained case by using two different \emph{active set strategies} and an adjusted step length. In active set strategies, inequality constraints are either treated as equality constraints (if they are close to 0) or neglected. This approach essentially combines the descent method on manifolds \cite{BFO2012} with the descent method for inequality constrained problems described in \cite{FS2000}.

In contrast to \cite{FV2016}, the method we propose computes a single Pareto optimal (or critical) point for each initial guess $x_0$. However, it is possible to interpret this approach as a discrete dynamical system and use set-oriented methods (see e.g.~\cite{DH1997}) to calculate its global attractor which contains the global Pareto set. It is also possible to use evolutionary approaches to globalize our method.

The outline of the article is as follows. In Section~\ref{sec:MO} we give a short introduction to multiobjective optimization and the steepest descent method for unconstrained MOPs. In Section~\ref{sec:Descent} we begin by generalizing this method to equality constraints and then proceed with the equality and inequality constrained case. We prove convergence for both cases. In Section~\ref{sec:Numerical_Results} we apply our method to an example to show its typical behavior and discuss ways to optimize it. In Section~\ref{sec:Conclusion} we draw a conclusion and discuss future work.

\section{Multiobjective optimization}
\label{sec:MO}
The goal of (unconstrained) multiobjective optimization is to minimize an objective function
\begin{equation*}
	F : \R^n \rightarrow \R^m, x = (x_1,...,x_n) \mapsto (F_1(x),...,F_m(x)).
\end{equation*}
Except in the case where all $F_i$ have the same minima, the definition of optimality from scalar-valued optimization does not apply. This is due to the loss of the total order for $m \geq 2$. We thus introduce the notion of \emph{dominance}.

\begin{definition}
	Let $v,w \in \R^m$. 
	\begin{enumerate}
		\item $v \leq w \quad :\Leftrightarrow \quad v_i \leq w_i \quad \forall i \in \{1,...,m\}$. Define $<$, $\geq$ and $>$ analogously.
		\item $v$ \emph{dominates} $w$, if $v \leq w \textrm{ and } v_i < w_i \textrm{ for some } i \in \{1,...,m\}.$
	\end{enumerate}
\end{definition}

The dominance relation defines a partial order on $\R^m$ such that we generally can not expect to find a minimum or infimum of $F(\R^n)$ with respect to that order. Instead, we want to find the set of \emph{optimal compromises}, the so-called \emph{Pareto set}.

\begin{definition} \label{def:dominate}
	\begin{enumerate}
		\item $x \in \R^n$ is \emph{locally Pareto optimal} if there is a neighborhood $U \subseteq \R^n$ of $x$ such that
		\begin{equation} \label{eq:localOptimal}
			\nexists y \in U : F(y) \text{ dominates } F(x).
		\end{equation}
		The set of locally Pareto optimal points is called the \emph{local Pareto set}.
		\item $x \in \R^n$ is \emph{globally Pareto optimal} if \eqref{eq:localOptimal} holds for $U = \R^n$. The set of globally Pareto optimal points is called the \emph{global Pareto set}.
		\item The \emph{local (global) Pareto front} is the image of the local (global) Pareto set under $F$.
	\end{enumerate}
\end{definition}
\noindent Note that the well-known Karush-Kuhn-Tucker (KKT) conditions from scalar-valued optimization also apply in the multiobjective situation \cite{KT51}.

Minimizing $F$ can now be defined as finding the Pareto set of $F$. The minimization of $F$ on a subset $\mathcal{N} \subseteq \R^n$ is defined the same way by replacing $\R^n$ in Definition \ref{def:dominate} with $\mathcal{N}$. For the constrained case, a point $x \in \R^n$ is called \emph{feasible} if $x \in \mathcal{N}$. In this paper we consider the case where $\mathcal{N}$ is given by a number of equality and inequality constraints. Thus, the general MOP we consider is of the form
\begin{equation} \label{MOP} \tag{MOP}
	\begin{aligned}
		& \underset{x \in \R^n}{\text{min}} && F(x), && \\
		& \text{s.t.} &&  H(x)  = 0, && \\
		&    && G(x) \leq 0,  &&
	\end{aligned}
\end{equation}
where $G: \R^n \rightarrow \R^{m_G}$ and $H: \R^n \rightarrow \R^{m_H}$ are differentiable.

For unconstrained MOPs, Fliege and Svaiter have proposed a descent method in \cite{FS2000} which we will now briefly summarize. Starting at a given point $x_0 \in \R^n$, the method generates a sequence $(x_l)_l \in \R^n$ with
\begin{equation*}
	F(x_{l+1}) < F(x_l) \quad \forall l \geq 0.
\end{equation*}
As the first-order necessary condition for local Pareto optimality they use
\begin{equation*}
	\textrm{im}(DF(x)) \cap (\R^{<0})^m = \emptyset ,
\end{equation*}
which is equivalent to
\begin{equation*}
	\nexists v \in \R^n: \nabla F_i(x) v < 0 \quad \forall i \in \{1,...,m\}.
\end{equation*}
Points $x \in \R^n$ satisfying this condition are called \emph{Pareto critical}. If $x$ is not Pareto critical, then $v \in \R^n$ is called a \emph{descent direction in} $x$ if $\nabla F_i(x) v < 0$ for all $i \in \{1,...,m\}$. Such a descent direction can be obtained via the following subproblem:
\begin{equation} \label{SP}
	\begin{aligned}
		& \underset{(v,\beta) \in \R^{n+1}}{\text{min}} &&  \beta + \frac{1}{2} \| v \|^2, && \\
		& \text{s.t.} &&  \nabla F_i(x) v \leq \beta &&\forall i \in \{1,...,m\}.
	\end{aligned} \tag{SP}
\end{equation}
By $\alpha(x) = \arg \min \eqref{SP}$ we denote the optimal value corresponding to $v$ for a fixed $x$. Problem~\eqref{SP} has the following properties.
\begin{lemma}
	\begin{enumerate}
		\item \eqref{SP} has a unique solution.
		\item If $x$ is Pareto critical, then $v(x) = 0$ and $\alpha(x) = 0$ .
		\item If $x$ is not Pareto critical, then $\alpha(x) < 0$.
		\item $x \mapsto v(x)$ and $x \mapsto \alpha(x)$ are continuous.
	\end{enumerate}
\end{lemma}
Thus, $x$ is Pareto critical iff $\alpha(x) = 0$. We want to use this descent method in a line search approach \cite{NW2006} and to this end choose a step length
\begin{equation*}
	t(x,v) = \max \left\{ s = \frac{1}{2^k} : k \in \mathbb{N}, \quad F(x + s v ) < F(x) + \sigma s DF(x)v \right\}
\end{equation*}
for some $\sigma \in (0,1)$. Using the descent direction and the step length described above, the sequence $(x_l)_l$ is calculated via the scheme
\begin{equation} \label{eq:OptStep}
	x_{l+1} = x_l + t(x_l,v(x_l)) v(x_l)
\end{equation}
for $l \geq 0$. As a convergence result the following was shown.
\begin{theorem}[\cite{FS2000}]
	Let $(x_l)_l$ be a sequence generated by the descent method described above. Then every accumulation point of $(x_l)_l$ is Pareto critical.
\end{theorem}
For $m = 1$ (i.e.~scalar-valued optimization), this method is reduced to the method of steepest descent where the step length satisfies the Armijo rule (see e.g.~\cite{NW2006}). In what follows, we will generalize this approach to constrained MOPs.

\section{Descent methods for the constrained case}
\label{sec:Descent}
In this section we propose two descent methods for constrained MOPs. We first define a descent direction and a step length for MOPs with equality constraints (ECs) and then use two different active set strategies to incorporate inequality constraints (ICs). We handle ECs similar to \cite{BFO2012} where the exponential map from Riemannian geometry is used to locally obtain new feasible points along a geodesic in a given (tangent) direction. Since in general, geodesics are obtained by solving an ordinary differential equation, it is not efficient to evaluate the exponential map by calculating geodesics. Instead, we will use \emph{retractions} \cite{AM2012} that can be thought of as first order approximations of the exponential map.

\subsection{Equality constraints}
\label{subsec:Descent_EC}
Our approach for handling ECs is of predictor-corrector type \cite{AG1990}. In the predictor step, we choose a descent direction along which the ECs are only mildly violated, and perform a step in that direction with a proper step length. In the corrector step, the resulting point is mapped onto the set satisfying the ECs. To this end, we choose a descent direction lying in the tangent space of the set given by the ECs. To ensure the existence of these tangent spaces, we make the following assumptions on $F$ and $H$.
\begin{itemize}
	\item[(A1)]~ \quad$F: \R^n \rightarrow \R^m$ is $C^2$ (two times continuously differentiable).
	\item[(A2)]~ \quad $H: \R^n \rightarrow \R^{m_H}$ is $C^2$ with regular value $0$ (i.e.~$\text{rk}(DH(x)) = m_H$ for all $x \in \R^n$ with $H(x) = 0$).
\end{itemize}
The assumption (A2) is also known as the \emph{linear independence constraint qualification} (LICQ) and is a commonly used constraint qualification (see e.g.~\cite{NW2006}). Let $\mathcal{M} := H^{-1}(\{0\})$ be the set of points satisfying the ECs. According to the Level Set Theorem (\cite{RS2013}, Example 1.6.4), $\mathcal{M}$ is a closed, $(n - m_H)$-dimensional $C^2$-submanifold of $\R^n$. It is Riemannian with the inner product $(v,w) \mapsto v^T w$. The tangent space in $x \in \mathcal{M}$ is $T_x \mathcal{M} = \text{ker}(DH(x))$ and the tangent bundle is
\begin{equation*}
	T\mathcal{M} = \bigcup\limits_{x \in \mathcal{M}} \{x\} \times \text{ker}(DH(x)).
\end{equation*}
Consequently, if we consider \eqref{MOP} without ICs, we may also write
\begin{equation*}
	\min\limits_{x \in \mathcal{M}} F(x), \tag{Pe}
\end{equation*}
Before generalizing the unconstrained descent method, we first have to extend the definition of Pareto criticality to the equality constrained case.

\begin{definition} \label{def:criticalEq}
	A point $x \in \mathcal{M}$ is \emph{Pareto critical}, if
	\begin{equation*}
		\nexists v \in \text{ker}(DH(x)): DF(x) v < 0.
	\end{equation*}
\end{definition}

This means that $x$ is Pareto critical iff it is feasible and there is no descent direction in the tangent space of $\mathcal{M}$ in $x$. We will show later (in Lemma \ref{lem:stepSizeExists}) that this is indeed a first-order necessary condition for local Pareto optimality. We now introduce a modification of the subproblem \eqref{SP} to obtain descent directions of $F$ in the tangent space of $\mathcal{M}$.
\begin{equation} \label{SPe}
	\begin{aligned}
		& \underset{(v,\beta) \in \R^{n+1}}{\text{min}} &&  \beta + \frac{1}{2} \| v \|^2 && \\
		& \text{s.t.} &&  \nabla F_i(x) v \leq \beta &&\forall i \in \{1,...,m\}, \\
		& &&  \nabla H_j(x) v = 0 &&\forall j \in \{1,...,m_H\}. 
	\end{aligned} \tag{SPe}
\end{equation}
An equivalent formulation is
\begin{equation*}
	\min\limits_{v \in \text{ker}(DH(x))} \left( \max\limits_{i}\nabla F_i(x) v + \frac{1}{2} \| v \|^2 \right) =: \alpha(x).
\end{equation*}
Since $v = 0$ (or $(v,\beta) = (0,0)$) is always a feasible point for this subproblem, we have $\alpha(x) \leq 0$ for all $x \in \mathcal{M}$. In the next three lemmas we generalize some results about the subproblem \eqref{SP} from the unconstrained case to \eqref{SPe}.

\begin{lemma} \label{lem:critical_alpha0}
	$x \in \mathcal{M}$ is Pareto critical iff $\alpha(x) = 0$.
\end{lemma}

Lemma \ref{lem:critical_alpha0} is easy to proof and shows that \eqref{SPe} can indeed be used to obtain a descent direction $v$ in the tanget space if $x$ is not Pareto critical. The next lemma will show that the solution of \eqref{SPe} is a unique convex combination of the projected gradients of $F_i$ onto the tangent space of $\mathcal{M}$.

\begin{lemma} \label{lem:alphaUnique_convexComb}
	\begin{enumerate}
		\item \eqref{SPe} has a unique solution.
		\item $v \in \R^n$ solves \eqref{SPe} iff there exist $\lambda_i \geq 0$ for $i \in I(x,v)$ so that
		\begin{equation*}
			v = - \sum_{i \in I(x,v)} \lambda_i \nabla_\mathcal{M} F_i(x), \quad \sum_{i \in I(x,v)} \lambda_i = 1,
		\end{equation*}
		where
		\begin{equation*}
			I(x,v) := \{i \in \{1,...,m\} : \nabla F_i(x) v = \max_{i \in \{1,...,m\}} \nabla F_i(x) v\}. 
		\end{equation*}
		$\nabla_\mathcal{M} F_i(x)$ denotes the gradient of $F_i$ as a function on $\mathcal{M}$, i.e.~$\nabla_\mathcal{M} F_i(x)$ is the projection of $\nabla F_i(x)$ onto $T_x \mathcal{M} = \text{ker}(DH(x))$.
	\end{enumerate}
\end{lemma}
\begin{proof}
	The first result follows from the strict convexity of the objective function in the equivalent formulation of \eqref{SPe}. The second result follows from theoretical results about subdifferentials. A detailed proof is shown in \cite[Lemma 4.1]{BFO2012}.
\end{proof}

By the last lemma, the function $v: \mathcal{M} \rightarrow \text{ker}(DH(x))$ which maps $x \in \mathcal{M}$ to the descent direction given by \eqref{SPe} is well defined. The following lemma shows that it is continuous.

\begin{lemma} \label{lem:avCont}
	The maps $v: \mathcal{M} \rightarrow \text{ker}(DH(x))$ and $\alpha: \mathcal{M} \rightarrow \R$ are continuous.
\end{lemma}
\begin{proof}
The continuity of $v$ is shown in \cite[Lemma 5.1]{BFO2012}. The continuity of $\alpha$ can be seen when decomposing $\alpha$ into the two following continuous maps:
	\begin{equation*}
		\begin{aligned}
			&\mathcal{M} \rightarrow \mathcal{M} \times \R^n, \quad q \mapsto (q,v(q)),\\
			&\mathcal{M} \times \R^n \rightarrow \R, \quad (q,w) \mapsto \max_i \nabla F_i(q) w + \frac{1}{2} \| w \|^2.
		\end{aligned}
	\end{equation*}
\end{proof}

Similar to \cite{FS2000}, it does not matter for the convergence theory if we take the exact solution of \eqref{SPe} or an inexact solution in the following sense.

\begin{definition}
	An \emph{approximate solution of \eqref{SPe} at $x \in \mathcal{M}$ with tolerance $\gamma \in (0,1]$} is a $v' \in \R^n$ such that 
	\begin{equation*}
		v' \in \text{ker}(DH(x)) \quad \textrm{ and } \quad \max\limits_{i}\nabla F_i(x) v' + \frac{1}{2} \| v' \|^2 \leq \gamma \alpha(x).
	\end{equation*} 
\end{definition}

If $x \in \mathcal{M}$ is not Pareto critical, we obviously still have $DF(x) v' < 0$ for all approximate solutions $v'$. Thus, approximate solutions are still descent directions in the tangent space. By setting the tolerance $\gamma$ to $1$, we again obtain the exact solution. For the relative error $\delta$ of the optimal value $\alpha$ of \eqref{SPe} we get
\begin{equation*}
	\delta(x) := \frac{| \max_i \nabla F_i(q) v' + \frac{1}{2} \| v' \|^2 - \alpha(x)|}{|\alpha(x)|} \leq \frac{|\gamma \alpha(x) - \alpha(x)|}{|\alpha(x)|} = 1 - \gamma.
\end{equation*}
By solving \eqref{SPe}, we can now compute descent directions in the tangent space of $\mathcal{M}$ at a feasible point $x \in \mathcal{M}$. In order to show that our method generates a decreasing sequence of feasible points, we will need some properties of the corrector step. To this end, we introduce a retraction map on $\mathcal{M}$ in the following definition.
 
\begin{definition} \label{def:projection}
	For $y \in \R^n$  consider the set
	\begin{equation*}
		P_\mathcal{M}(y) := \{ x \in \mathcal{M}: \| x - y \| = \min\limits_{x \in \mathcal{M}} \| x - y \| \}.
	\end{equation*}
	According to \cite[Lemma 3.1]{AM2012}, for every $x \in \mathcal{M}$ there exists some $\epsilon > 0$ such that $P_\mathcal{M}(y)$ contains only one element for all $y \in U_\epsilon(x) := \{ y \in \R^n : \| y - x \| < \epsilon \}$. We can thus define the map
	\begin{equation*}
		\pi: \mathcal{V} \rightarrow \mathcal{M}, \quad y \mapsto P_\mathcal{M}(y)
	\end{equation*}
	in a neighborhood $\mathcal{V}$ of $\mathcal{M}$. It was also shown that this map is $C^1$. According to \cite[Proposition 3.2]{AM2012}, the map
	\begin{equation*}
		R_\pi: T\mathcal{M} \rightarrow \mathcal{M}, \quad (x,v) \mapsto \pi(x+v)
	\end{equation*} 
	is a $C^1$\emph{-retraction on $\mathcal{M}$}, i.e.~for all $x \in \mathcal{M}$ there is a neighborhood $\mathcal{U}$ of $(x,0) \in T \mathcal{M}$ such that
	\begin{enumerate}
		\item the restriction of $R_\pi$ onto $\mathcal{U}$ is $C^1$,
		\item $R_\pi(q,0) = q \quad \forall (q,0) \in \mathcal{U}$,
		\item $DR_\pi(q,\cdot)(0) = \textrm{id}_{\text{ker}(DH(q))} \quad \forall (q,0) \in \mathcal{U}$.
	\end{enumerate}
	(More precisely, it was shown that $\pi$ is $C^{k-1}$ and $R_\pi$ is a $C^{k-1}$-retraction on $\mathcal{M}$ if $\mathcal{M}$ is $C^k$.) In all $y \in \R^n$ where $\pi$ is undefined, we set $\pi(y)$ to be some element of $P_{\mathcal{M}}(y)$. This will not matter for our convergence results since we only need the retraction properties when we are near $\mathcal{M}$.
\end{definition}

Using $v$ and $\pi$, we can now calculate proper descent directions for the equality constrained case. As in all line search strategies, we furthermore have to choose a step length which assures that the new feasible point is an improvement over the previous one and that the resulting sequence converges to a Pareto critical point. To this end -- for $x \in \mathcal{M}$ and $v \in \text{ker}(DH(x))$ with $DF(x)v < 0$ given -- consider the step length
\begin{equation} \label{eq:stepSize}
	t = \beta_0 \beta^k,
\end{equation}
where
\begin{equation*}
	k := \min\{k \in \mathbb{N}: F(\pi(x + \beta_0 \beta^k v)) < F(x) + \sigma \beta_0 \beta^k DF(x) v \}
\end{equation*}
with $\beta_0 > 0$, $\beta \in (0,1)$ and $\sigma \in (0,1)$. To show that such a $k$ always exists, we first require the following lemma.

\begin{lemma} \label{lem:diffQuot}
	Let $n_1, n_2 \in \mathbb{N}$, $U \subseteq \R^{n_1}$ open and $f: U \rightarrow \R^{n_2}$ $C^2$. Let $(x_k)_k \in U$, $(v_k)_k \in \R^{n_1}$ and $(t_k)_k \in \R^{>0}$ such that $\lim\limits_{k \rightarrow \infty} x_k = x \in U$, $\lim\limits_{k \rightarrow \infty} v_k = v$ and $\lim\limits_{k \rightarrow \infty} t_k = 0$. Then
	\begin{equation*}
		\lim\limits_{k \rightarrow \infty} \frac{f(x_k + t_k v_k) - f(x_k)}{t_k} = Df(x) v.
	\end{equation*}
	If additionally  $Df(x) v < 0$, then for all $\sigma \in (0,1)$ there exists some $K \in \mathbb{N}$ such that
	\begin{equation*}
		f(x_k + t_k v_k) < f(x_k) + \sigma t_k Df(x) v \quad \forall k \geq K.
	\end{equation*}
\end{lemma}
\begin{proof}
	The proof follows by considering the Taylor series expansion of $f$.
\end{proof}

With the last lemma we can now show the existence of the step length \eqref{eq:stepSize}.

\begin{lemma} \label{lem:stepSizeExists}
	Let $v \in \text{ker}(DH(x))$ with $DF(x)v < 0$. Let $\beta_0 > 0$, $\beta \in (0,1)$ and $\sigma \in (0,1)$. Then
	\begin{align} \label{eq:setArmijoIneq}
		\{k \in \mathbb{N}: F(\pi(x + \beta_0 \beta^k v)) < F(x) + \sigma \beta_0 \beta^k DF(x) v \} \neq \emptyset
	\end{align}
	and there exists some $K \in \mathbb{N}$ such that this set contains all $k > K$. Particularly
	\begin{equation*}
		\lim\limits_{k \rightarrow \infty}\frac{F(\pi(x + \beta_0 \beta^k v)) - F(x)}{\beta_0 \beta^k} = DF(x)v.
	\end{equation*}
\end{lemma}
\begin{proof} 
	We use Lemma \ref{lem:diffQuot} with $x_k = x$, $t_k = \beta_0 \beta^k$ and $v_k = \frac{\pi(x + \beta_0 \beta^k v) - x}{\beta_0 \beta^k}$. We have to show that
	\begin{equation*}
		\lim\limits_{k \to \infty} \frac{\pi(x + \beta_0 \beta^k v) - x}{\beta_0 \beta^k} = v.
	\end{equation*}
	Since $R_\pi$ is a retraction, we have
	\begin{equation*}
	\lim\limits_{k \to \infty} \frac{\pi(x + \beta_0 \beta^k v) - x}{\beta_0 \beta^k} = \lim\limits_{k \to \infty} \frac{R_\pi(x,\beta_0 \beta^k v) - R_\pi(x,0)}{\beta_0 \beta^k} = DR_\pi(x,\cdot)(0)(v) = v.
	\end{equation*}
\end{proof}

The inequality in the set \eqref{eq:setArmijoIneq} is called the \emph{Armijo inequality}. In particular, the last lemma shows that Pareto criticality (according to Definition~\ref{def:criticalEq}) is a necessary condition for Pareto optimality since $\pi(x + \beta_0 \beta^k v) \rightarrow x$ for $k \rightarrow \infty$ and $F(\pi(x + \beta_0 \beta^k v)) < F(x)$ for $k$ large enough if $v$ is a descent direction.

The descent method for equality constrained MOPs is summarized in Algorithm~\ref{algo:eq}. To obtain a feasible initial point $x_0 \in \mathcal{M}$, we evaluate $\pi$ at the initial point $x \in \R^n$.

\begin{algorithm} 
	\caption{(Descent method for equality constrained MOPs)}
	\label{algo:eq}
	\begin{algorithmic}[1]
		\Require $x \in \R^n$, $\gamma \in (0,1]$, $\beta_0 > 0$, $\beta \in (0,1)$, $\sigma \in (0,1)$.
		\State Compute $x_0 = \pi(x)$.
		\For{$l = 0$, $1$, ...}
			\State Compute an approximate solution $v_l$ of \eqref{SPe} at $x_l$ with tolerance $\gamma$.
			\If{$x_l$ is Pareto critical (i.e.~$\alpha(x_l) = 0$)}
				\State Stop.
			\Else
				\State Compute $t_l$ as in \eqref{eq:stepSize} for $x_l$ and $v_l$ with parameters $\beta_0$, $\beta$ and $\sigma$.
				\State Set $x_{l+1} = \pi(x_l + t_l v_l)$.
			\EndIf
		\EndFor
	\end{algorithmic}
\end{algorithm}

Similar to \cite[Theorem 5.1]{BFO2012}, we have the following convergence result. 

\begin{theorem} \label{thm:convEq}
	Let $(x_l)_l \in \mathcal{M}$ be a sequence generated by Algorithm \ref{algo:eq}. Then $(x_l)_l$ is either finite (and the last element of the sequence is Pareto critical) or each accumulation point of $(x_l)_l$ is Pareto critical. 
\end{theorem}
\begin{proof}
	A sequence generated by Algorithm \ref{algo:eq} can only be finite if $\alpha(x_l) = 0$, which means $x_l$ is Pareto critical. Let $(x_l)_l$ from now on be infinite, so no element of the sequence is Pareto critical. \\
	Let $\bar{x} \in \mathcal{M}$ be an accumulation point of $(x_l)_l$. By construction of Algorithm \ref{algo:eq}, each component of $(F(x_l))_l$ is monotonically decreasing. Therefore, since $F$ is continuous, we have $\lim\limits_{l \rightarrow \infty} F(x_l) = F(\bar{x})$. By our choice of the descent direction and the step length \eqref{eq:stepSize} we have
	\begin{equation*}
		F(x_{l+1}) - F(x_l) < \sigma t_l DF(x) v_l < 0 \quad \forall l \in \mathbb{N}.
	\end{equation*}
	It follows that
	\begin{equation*}
		\lim\limits_{l \rightarrow \infty} t_l DF(x_l) v_l = 0.
	\end{equation*}
	Let $(l_s)_s \in \mathbb{N}$ be a sequence of indices so that $\lim\limits_{s \rightarrow \infty} x_{l_s} = \bar{x}$. We have to consider the following two possibilities:
	\begin{enumerate}
		\item $\underset{l \rightarrow \infty}{\textrm{limsup }} t_l > 0$
		\item $\lim\limits_{l \rightarrow \infty} t_l = 0$
	\end{enumerate}
	\textbf{Case 1:} It follows that
	\begin{equation*}
		\lim\limits_{l \rightarrow \infty} DF(x_l) v_l = 0,
	\end{equation*}
	and particularly
	\begin{equation*}
		\lim\limits_{l \rightarrow \infty} \max\limits_{i} \nabla F_i(x_l) v_l = 0.
	\end{equation*}
	Since $v_l$ is an approximate solution for some $\gamma \in (0,1]$, we have
	\begin{equation*}
		\max\limits_{i} \nabla F_i(x_l) v_l + \frac{1}{2} \| v_l \|^2 \leq \gamma \alpha(x_l) < 0 \quad \forall l \in \mathbb{N},
	\end{equation*}
	and it follows that
	\begin{equation*}
		\lim\limits_{l \rightarrow \infty} \left( \max\limits_{i} \nabla F_i(x_l) v_l + \frac{1}{2} \| v_l \|^2 \right) = 0.
	\end{equation*}
	Since $\alpha$ is continuous, this means
	\begin{equation*}
		0 = \lim\limits_{s \rightarrow \infty} \alpha(x_{l_s}) = \alpha(\bar{x}),
	\end{equation*}
	and $\bar{x}$ is Pareto critical.\\
	\textbf{Case 2:} It is easy to see that the sequence $(v_{l_s})_s$ of approximate solutions is bounded. Therefore, it is contained in a compact set and possesses an accumulation point $\bar{v}$. Let $(l_u)_u \in \mathbb{N}$ be a subsequence of $(l_s)_s$ with $\lim\limits_{u \rightarrow \infty} v_{l_u} = \bar{v}.$
	Since all $v_{l_u}$ are approximate solutions to \eqref{SPe}, we have
	\begin{equation*}
		\max_i \nabla F_i(x_{l_u}) v_{l_u} \leq \max_i \nabla F_i(x_{l_u}) v_{l_u} + \frac{1}{2} \| v_{l_u} \|^2 \leq \gamma \alpha(x_{l_u}) < 0.
	\end{equation*}
	Letting $u \rightarrow \infty$ and considering the continuity of $\alpha$, we thus have
	\begin{equation}
		\max_i \nabla F_i(\bar{x}) \bar{v} \leq \max_i \nabla F_i(\bar{x}) \bar{v} + \frac{1}{2} \| \bar{v} \|^2 \leq \gamma \alpha(\bar{x}) \leq 0. \label{eq:thmConvEq1}
	\end{equation}
	Since $\lim\limits_{l \rightarrow \infty} t_l = 0$, for all $q \in \mathbb{N}$ there exists some $N \in \mathbb{N}$ such that
	\begin{equation*}
		t_{l_u} < \beta_0 \beta^q \quad \forall u > N,
	\end{equation*} 
	and therefore -- as a consequence of the definition of the step length \eqref{eq:stepSize} -- we have
	\begin{equation*}
		F(\pi(x_{l_u} + \beta_0 \beta^q v_{l_u})) \nless F(x_{l_u}) + \sigma \beta_0 \beta^q DF(x_{l_u}) v_{l_u} \quad \forall u > N.
	\end{equation*}
	Thus, there has to be some $i \in \{1,...,m\}$ such that
	\begin{equation*}
		F_i(\pi(x_{l_u} + \beta_0 \beta^q v_{l_u})) \geq F_i(x_{l_u}) + \sigma \beta_0 \beta^q \nabla F_i(x_{l_u}) v_{l_u}
	\end{equation*}
	holds for infinitely many $u > N$. Letting $u \rightarrow \infty$ we get
	\begin{equation*}
		F_i(\pi(\bar{x} + \beta_0 \beta^q \bar{v})) \geq F_i(\bar{x}) + \sigma \beta_0 \beta^q \nabla F_i(\bar{x}) \bar{v},
	\end{equation*}
	and therefore
	\begin{equation}
		\frac{F_i(\pi(\bar{x} + \beta_0 \beta^q \bar{v})) - F_i(\bar{x})}{\beta_0 \beta^q} \geq \sigma \nabla F_i(\bar{x}) \bar{v}. \label{eq:thmConvEq2}
	\end{equation}
	There has to be at least one $i \in \{1,...,m\}$  such that inequality \eqref{eq:thmConvEq2} holds for infinitely many $q \in \mathbb{N}$. By Lemma \ref{lem:stepSizeExists} and letting $q \rightarrow \infty$, we have
	\begin{equation*}
		\nabla F_i(\bar{x}) \bar{v} \geq \sigma \nabla F_i(\bar{x}) \bar{v}
	\end{equation*}
	and thus, $\nabla F_i(\bar{x}) \bar{v} \geq 0$. Combining this with inequality \eqref{eq:thmConvEq1}, we get
	\begin{equation*}
		\max_i \nabla F_i(\bar{x}) \bar{v} = 0.
	\end{equation*}
	Therefore, we have
	\begin{equation*}
		\gamma \alpha(\bar{x}) \geq \max_i \nabla F_i(\bar{x}) \bar{v} + \frac{1}{2} \| \bar{v} \|^2 = 0,
	\end{equation*}
	and $\alpha(\bar{x}) = 0$ by which $\bar{x}$ is Pareto critical.
\end{proof}

To summarize, the above result yields the same convergence as in the unconstrained case. 

\begin{remark} \label{rem:modConvEq}
	Observe that the proof of Theorem \ref{thm:convEq} also works for slightly more general sequences than the ones generated by Algorithm \ref{algo:eq}. (This will be used in a later proof and holds mainly due to the fact that we consider subsequences in the proof of Theorem~\ref{thm:convEq}.) \\
	Let $(x_l)_l \in \mathcal{M}$, $K \subseteq \mathbb{N}$ with $| K | = \infty$ such that 
	\begin{itemize}
		\item $(F_i(x_l))_l$ is monotonically decreasing for all $i \in \{1,...,m\}$ and
		\item the step from $x_k$ to $x_{k+1}$ was realized by an iteration of Algorithm \ref{algo:eq} for all $k \in K$.
	\end{itemize}
	Let $\bar{x} \in \mathcal{M}$ such that there exists a sequence of indices $(l_s)_s \in K$ with $\lim_{s \rightarrow \infty} x_{l_s} = \bar{x}$. Then $\bar{x}$ is Pareto critical.
\end{remark}


\subsection{Equality and inequality constraints}
\label{subsec:Descent_EC_IC}
In order to incorporate inequality constraints (ICs), we consider two different active set strategies in which ICs are either \emph{active} (and hence treated as ECs) or otherwise neglected. An inequality constraint will be considered active if its value is close to zero. The first active set strategy is based on \cite[Section 8]{FS2000}. There, the active ICs are treated as additional components of the objective function, so the values of the active ICs decrease along the resulting descent direction. This means that the descent direction points into the feasible set with respect to the ICs and the sequence moves away from the boundary. The second active set strategy treats active ICs as additional ECs. Thus, the active ICs stay active and the sequence stays on the boundary of the feasible set with respect to the ICs. Before we can describe the two strategies rigorously, we first have to define when an inequality is active. Furthermore, we have to extend the notion of Pareto criticality to ECs and ICs.

As in the equality constrained case, we have to make a few basic assumptions about the objective function and the constraints.
\begin{itemize}
	\item[] \textbf{(A1):} $F: \R^n \rightarrow \R^m$ is $C^2$.
	\item[] \textbf{(A2):} $H: \R^n \rightarrow \R^{m_H}$ is $C^2$ with regular value $0$.
	\item[] \textbf{(A3):} $G: \R^n \rightarrow \R^{m_G}$ is $C^2$.
\end{itemize}

By (A2), $\mathcal{M} := H^{-1}(\{0\})$ has the same manifold structure as in the equality constrained case. For the remainder of this section we consider MOPs of the form
\begin{equation} \label{P}
	\begin{aligned}
		& \underset{x \in \mathcal{M}}{\text{min}} &&  F(x), \\
		& \text{s.t.} &&  G_i(x) \leq 0 \quad \forall i \in \{1,...,m_G\},
	\end{aligned} \tag{P}
\end{equation}
which is equivalent to \eqref{MOP}. The set of feasible points is $\mathcal{N} := \mathcal{M} \cap G^{-1}((\R^{\leq 0})^{m_G})$ and can be thought of as a manifold with a boundary (and corners). In order to use an active set strategy we first have to define the active set.

\begin{definition}
	For $\epsilon \geq 0$ and $x \in \mathcal{M}$ set
	\begin{equation*}
		I_\epsilon(x) := \{i \in \{1,...,m_G\}: G_i(x) \geq -\epsilon \}.
	\end{equation*}
	The component functions of $G$ with indices in $I_\epsilon(x)$ are called \emph{active} and $I_\epsilon(x)$ is called the \emph{active set in $x$ (with tolerance $\epsilon$)}. A boundary $G_i^{-1}(\{0\})$ is called \emph{active} if $i \in I_\epsilon(x)$.
\end{definition}

In other words, the $i$-th inequality being active at $x \in \mathcal{N}$ means that $x$ is close to the $i$-th boundary $G_i^{-1}(\{0\})$. Before defining Pareto criticality for problem \eqref{P}, we need an additional assumption on $G$ and $H$. Let
\begin{equation*}
	L(x) := \{ v \in \text{ker}(DH(x)) : \nabla G_i(x) v \leq 0 \quad \forall i \in I_0(x) \}
\end{equation*}
be the \emph{linearized cone at} $x$.
\begin{itemize}
	\item[] \textbf{(A4):} The interior of $L(x)$ is non-empty for all $x \in \mathcal{N}$, i.e.
	\begin{equation*}
		L^\circ(x) = \{ v \in \text{ker}(DH(x)) : \nabla G_i(x) v < 0 \quad \forall i \in I_0(x) \} \neq \emptyset.
	\end{equation*}
\end{itemize}
For example, this assumption eliminates the case where $G^{-1}((\R^{\leq 0})^{m_G})$ has an empty interior. We will use the following implication of (A4):

\begin{lemma} \label{lem:A4equiv}
	(A4) holds iff for all $\delta > 0$, $x \in \mathcal{N}$ and $v \in \text{ker}(DH(x))$ with
	\begin{equation*}
		\nabla G_i(x) v \leq 0 \quad \forall i \in I_0(x),
	\end{equation*}
	there is some $w \in U_{\delta}(v) \cap \text{ker}(DH(x))$ with 
	\begin{equation*}
		\nabla G_i(x) w < 0 \quad \forall i \in I_0(x).
	\end{equation*}
\end{lemma}
\begin{proof}
	Assume that there is some $v \in L(x)$ and $\delta > 0$ such that 
	\begin{equation*}
		\nexists \epsilon \in U_\delta(0) \cap \text{ker}(DH(x)): \nabla G_i(x) (v+\epsilon) < 0 \quad \forall i \in I_0(x).
	\end{equation*}
	It follows that
	\begin{equation*}
		\forall \epsilon \in U_\delta(0) \cap \text{ker}(DH(x)) \hspace{3pt} \exists j \in I_0(x): \nabla G_{j}(x) \epsilon \geq 0,
	\end{equation*}
	and thus,
	\begin{equation*}
		\forall w \in \text{ker}(DH(x)) \hspace{3pt} \exists j \in I_0(x): \nabla G_{j}(x) w \geq 0.
	\end{equation*}
	So we have $L^\circ(x) = \emptyset$. On the other hand, the right-hand side of the equivalence stated in this lemma obviously can not hold if $L^\circ(x) = \emptyset$, which completes the proof.
\end{proof}

We will now define Pareto criticality for MOPs with ECs and ICs. Similar to scalar-valued optimization, Pareto critical points are feasible points for which there exists no descent direction pointing into or alongside the feasible set $\mathcal{N}$.

\begin{definition}
	A point $x \in \mathcal{N}$ is \emph{Pareto critical}, if
	\begin{equation*}
		\nexists v \in \text{ker}(DH(x)): DF(x) v < 0 \quad \textrm{and} \quad \nabla G_i(x) v \leq 0 \quad \forall i \in I_0(x).
	\end{equation*}
\end{definition}

By Lemma \ref{lem:A4equiv}, the condition in the last definition is equivalent to
\begin{equation*}
	\nexists v \in \text{ker}(DH(x)): DF(x) v < 0 \quad \textrm{and} \quad \nabla G_i(x) v < 0 \quad \forall i \in I_0(x).
\end{equation*}
We will show in Lemma \ref{lem:stepSizeS1Exists} that Pareto criticality is indeed a first-order necessary condition for Pareto optimality for the MOP \eqref{P}. We will now look at the first active set strategy.

\subsubsection{Strategy 1 -- Active inequalities as additional objectives}
\label{subsubsec:Descent_EC_IC_S1}
If we consider the active constraints in the subproblem \eqref{SPe} as additional components of the objective function $F$, then we obtain the following subproblem. For given $x \in \mathcal{N}$ and $\epsilon > 0$:
\begin{equation} \label{SP1}
	\begin{aligned}
		& \alpha_1(x,\epsilon) := && \underset{(v,\beta) \in \R^{n+1}}{\text{min}} &&  \beta + \frac{1}{2} \| v \|^2  && \\
		& && \text{s.t.} &&  \nabla F_i(x) v \leq \beta &&\forall i \in \{1,...,m\}, \\
		& && && \nabla G_l(x) v \leq \beta &&\forall l \in I_\epsilon(x), \\
		& && && \nabla H_j(x) v = 0 &&\forall j \in \{1,...,m_H\}.
	\end{aligned} \tag{SP1}
\end{equation}

Due to the incorporated active set strategy, the solution $v$ and the optimal value $\alpha_1(x,\epsilon)$ (which now additionally depends on the tolerance $\epsilon$) of this subproblem do in general not depend continuously on $x$.

\begin{remark} \label{rem:propertiesAlphaS1}
	The following two results for \eqref{SPe} can be translated to \eqref{SP1}.
	\begin{enumerate}
		\item By Lemma \ref{lem:critical_alpha0} and (A4) we obtain that a point $x$ is Pareto critical iff $\alpha_1(x,0) = 0$. Furthermore, a solution of \eqref{SP1} is a descent direction of $F$ if the corresponding optimal value $\alpha_1$ is negative.
		\item Using Lemma \ref{lem:alphaUnique_convexComb} with the objective function extended by the active inequalities, we obtain uniqueness of the solution of \eqref{SP1}. 
	\end{enumerate}
\end{remark}

Using \eqref{SP1} instead of \eqref{SPe} and modifying the step length yields Algorithm \ref{algo:S1} as a descent method for the MOP \eqref{P}. Similar to the equality constrained case, we first have to calculate a feasible point to start our descent method.
\begin{algorithm}[t]
	\caption{(Descent method for equality and inequality constrained MOPs, Strategy 1)} \label{algo:S1}
	\begin{algorithmic}[1]
		\Require $x \in \R^n$, $\beta_0 > 0$, $\beta \in (0,1)$, $\epsilon > 0$, $\sigma \in (0,1)$.
		\State Compute some $x_0 \in \mathcal{N}$ with $\| x - x_0 \| = \min\{\| x - x_0 \| : x_0 \in \mathcal{N \}}$.
		\For{$l = 0$, $1$, ...}
			\State Identify the active set $I_\epsilon(x_l)$.
			\State Compute the solution $v_l$ of \eqref{SP1} at $x_l$ with $\epsilon$.
			\If{$x_l$ is Pareto critical (i.e.~$\alpha_1(x_l,\epsilon) = 0$)}
				\State Stop.
			\Else
				\State Compute the step length $t_l$ as in \eqref{eq:stepSize} for $x_l$ and $v_l$ with parameters $\beta_0$, $\beta$ and $\sigma$.
				\If{$\pi(x_l + t_l v_l) \notin \mathcal{N}$}
					\State Compute the smallest $k$ so that
						\begin{equation*}
							F(\pi(x_l + \beta_0 \beta^k v_l)) < F(x_l) + \sigma \beta_0 \beta^k DF(x_l) v_l
						\end{equation*}
					\hspace{37pt} and $\pi(x_l + \beta_0 \beta^k v_l) \in \mathcal{N}$.
					\State Choose some $t_l \in [\beta_0 \beta^k, \beta_0 \beta^{k_l}]$, so that $\pi(x_l + t_l v_l) \in \mathcal{N}$ and the Armijo 
					\Statex \hspace{37pt} condition hold.
				\EndIf
				\State Set $x_{l+1} = \pi(x_l + t_l v_l)$.
			\EndIf
		\EndFor
	\end{algorithmic}
\end{algorithm}
We have modified the step length such that in addition to the Armijo condition, the algorithm verifies whether the resulting point is feasible with respect to the inequality constraints. In order to show that this algorithm is well defined, we have to ensure that this step length always exists.

\begin{lemma} \label{lem:stepSizeS1Exists}
	Let $x \in \mathcal{N}$, $\beta_0 > 0$, $\beta \in (0,1)$, $\sigma \in (0,1)$, $\epsilon \geq 0$ and $v \in \text{ker}(DH(x))$ with
	\begin{equation*}
		DF(x) v < 0 \quad \textrm{and} \quad \nabla G_i(x) v < 0 \quad \forall i \in I_\epsilon(x).
	\end{equation*}
	Then there exists some $K \in \mathbb{N}$ such that
	\begin{equation*}
		F(\pi(x + \beta_0 \beta^k v)) < F(x) + \sigma \beta_0 \beta^k DF(x)v \quad \textrm{and} \quad \pi(x + \beta_0 \beta^k v) \in \mathcal{N} \quad \forall k > K.
	\end{equation*}
	Particularly, Pareto criticality is a first-order necessary condition for local Pareto optimality.
\end{lemma}
\begin{proof}
	Lemma \ref{lem:stepSizeExists} shows the existence of some $K \in \mathbb{N}$ for which the first condition holds for all $k > K$. We assume that the second condition is violated for infinitely many $k \in \mathbb{N}$. Then we have $G(x) \leq 0$ and 
	\begin{equation}
		G_j(\pi(x + \beta_0 \beta^k v)) > 0 \label{eq:LemStepSizeS1ExistsEq1}
	\end{equation}
	for arbitrarily large $k \in \mathbb{N}$ and some $j \in \{1,...,m_G\}$. Since $G$ and $\pi$ are continuous, we get $G_j(x) = 0$ and $j \in I_\epsilon(x)$. Using inequality \eqref{eq:LemStepSizeS1ExistsEq1} combined with Lemma \ref{lem:stepSizeExists} with $F = G_j$, we get
	\begin{equation*}
		0 \leq \lim\limits_{k \rightarrow \infty} \frac{G_j(\pi(x + \beta_0 \beta^k v)) - G_j(x)}{\beta_0 \beta^k} = \nabla G_j(x) v.
	\end{equation*}
	This contradicts our prerequisites.
\end{proof}

Since the step length always exists, we know that Algorithm \ref{algo:S1} generates a sequence $(x_l)_l \in \mathcal{N}$ with $F(x_{l+1}) < F(x_l)$. We will now prove a convergence result of this sequence (Theorem \ref{thm:convS1}). The proof is essentially along the lines of \cite[Section 8]{FS2000} and is based on the observation that the step lengths in the algorithm can not become arbitrarily small if $\alpha_1(x,\epsilon) < \rho < 0$ holds (for some $\rho < 0$) for all $x$ in a compact set. To prove this, we have to show the existence of a positive lower bound for step lengths violating the two requirements in Step 10 of Algorithm \ref{algo:S1}. To this end, we first need the following technical result.

\begin{lemma} \label{lem:lipIneq}
	Let $f: \R^n \rightarrow \R$ be a continuously differentiable function so that $\nabla f : \R^n \rightarrow \R^n$ is Lipschitz continuous with a constant $L > 0$. Then for all $x,y \in \R^n$ we have
	\begin{equation*}
		f(y) \leq f(x) + \nabla f(x) (y-x) + \frac{1}{2} L \| y-x \|^2.
	\end{equation*}
\end{lemma}
\begin{proof}
	The proof follows by considering 
	\begin{equation*}
		f(y) = f(x) + \nabla f(x) (y - x)  + \int_0^1 (\nabla f(x + t(y-x)) + \nabla f(x))(y - x) dt
	\end{equation*}
	and by using the Lipschitz continuity of $\nabla f$ to get an upper bound for the integral.
\end{proof}

The next lemma will show that there exists a lower bound for the step lengths that violate the first requirement in Step 10 of Algorithm \ref{algo:S1}.

\begin{lemma} \label{lem:stepSizeS1Bound1}
	Let $\mathcal{K} \subseteq \R^n$ be compact, $\mathcal{V} \subseteq \R^n$ be a closed sphere around $0$, $\delta > 0$ and $\sigma \in (0,1)$. Then there is some $T > 0$ so that
	\begin{equation}
		F(\pi(x + tv)) \leq F(x) + t \sigma DF(x) v \label{eq:LemStepSizeS1Bound1Eq1}
	\end{equation}
	holds for all $x \in \mathcal{K} \cap \mathcal{N}$, $v \in \mathcal{V}$ with $DF(x) v \leq 0$ and
	\begin{equation*}
		t \in \left[ 0, \min \left( T, -\frac{2 (1 - \sigma) \max_i (\nabla F_i(x)v) + \delta}{L^3 \| v \|^2} \right) \right],
	\end{equation*}
	where $L$ is a Lipschitz constant on $\mathcal{K}$ for $\pi$ and $\nabla F_i$, $i \in \{1,...,m\}$.
\end{lemma}
\begin{proof}
	Since $\pi$ and all $\nabla F_i$ are continuously differentiable and $\mathcal{K}$ is compact, there is a (common) Lipschitz constant $L$ (cf. \cite[Proposition 2 and 3]{H2003}). Thus, by Lemma \ref{lem:lipIneq} we have
	\begin{equation*}
		\begin{aligned}
			F_i(\pi(x+tv)) &= F_i(x + (\pi(x+tv) - x)) \\
			&\leq F_i(x) + \nabla F_i(x) (\pi(x + tv) - x) + \frac{1}{2} L \| \pi(x+tv) - x \|^2.
		\end{aligned}
	\end{equation*}
	This means that \eqref{eq:LemStepSizeS1Bound1Eq1} holds if for all $i \in \{1,...,m\}$ we have
	\begin{equation*}
		F_i(x) + \nabla F_i(x) (\pi(x + tv) - x) + \frac{1}{2} L \| \pi(x+tv) - x \|^2 \leq F_i(x) + t \sigma \nabla F_i(x) v,
	\end{equation*}
	which is equivalent to
	\begin{equation}
		t  \nabla F_i(x) \left( \frac{\pi(x+tv)-x}{t} - \sigma v \right) \leq - \frac{1}{2} L \| \pi(x+tv) - x \|^2. \label{eq:LemStepSizeS1Bound1Eq2}
	\end{equation}
	Since $\pi$ is Lipschitz continuous we have
	\begin{equation*}
		- \frac{1}{2} L \| \pi(x+tv) - x \|^2 \geq - \frac{1}{2} L^3 t^2 \| v \|^2,
	\end{equation*}
	and \eqref{eq:LemStepSizeS1Bound1Eq2} holds if
	\begin{equation*}
		t \nabla F_i(x) \left( \frac{\pi(x+tv)-x}{t} - \sigma v \right) \leq - \frac{1}{2} L^3 t^2 \| v \|^2,
	\end{equation*}
	which is equivalent to
	\begin{equation}
		\begin{aligned}
			& t \leq -\frac{2 \nabla F_i(x) (\frac{\pi(x+tv)-x}{t} - \sigma v) }{L^3 \| v \|^2}  \\
			& = -\frac{2 (1 - \sigma) \nabla F_i(x)v + 2 \nabla F_i(x) (\frac{\pi(x+tv)-x}{t} - v) }{L^3 \| v \|^2} \label{eq:LemStepSizeS1Bound1Eq3}
		\end{aligned}
	\end{equation}
	for all $i \in \{1,...,m\}$. Since $\mathcal{K}$ and $\mathcal{V}$ are compact and $\lim\limits_{t \rightarrow 0} \frac{\pi(x+tv)-x}{t} = v$ (cf. the proof of Lemma \ref{lem:stepSizeExists}), we know by continuity that for all $\delta' > 0$, there exists some $T' > 0$ such that 
	\begin{equation*}
		\sup_{x \in \mathcal{K} \cap \mathcal{N},v \in \mathcal{V},t \in (0,T']} \Big\| \frac{\pi(x+tv)-x}{t} - v \Big\| < \delta'.
	\end{equation*}
	Since all $\| \nabla F_i \|$ are continuous on $\mathcal{K}$ and therefore bounded, by the Cauchy-Schwarz inequality there exists some $T > 0$ such that
	\begin{equation*}
		\left| 2 \nabla F_i(x) \left( \frac{\pi(x+tv)-x}{t} - v \right) \right| < \delta
	\end{equation*}
	for all $i \in \{1,...,m\}$, $x \in \mathcal{K} \cap \mathcal{N}$, $v \in \mathcal{V}$ and $t \in [0,T]$. Combining this with inequality \eqref{eq:LemStepSizeS1Bound1Eq3} completes the proof.
\end{proof}

The following lemma shows that there is a lower bound for the second requirement in Step 10 of Algorithm \ref{algo:S1}. This means that if we have a direction $v$ pointing inside the feasible set given by the active inequalities, then we can perform a step of a certain length in that direction without violating any inequalities.

\begin{lemma} \label{lem:stepSizeS1Bound2}
	Let $\mathcal{K} \subseteq \R^n$ be compact, $\mathcal{V} \subseteq \R^n$ be a closed sphere around $0$, $\delta > 0$ and $L$ be a Lipschitz constant of $G_i$ and $\nabla G_i$ on $\mathcal{K}$ for all $i$. Let $x \in \mathcal{K} \cap \mathcal{N}$, $v \in \mathcal{V}$, $\epsilon > 0$ and $\rho < 0$, such that
	\begin{equation*}
		\nabla G_i(x) v \leq \rho \quad \forall i \in I_\epsilon(x).
	\end{equation*}
	Then there exists some $T > 0$ such that
	\begin{equation*}
		\pi(x + tv) \in \mathcal{N} \quad \forall t \in \left[0, \min\left(T, \frac{\epsilon}{L (\| v \| + \delta)} , -\frac{2 (\rho + \delta)}{L(\| v \| + \delta)^2} \right)\right].
	\end{equation*}
\end{lemma}
\begin{proof}
	Since $G_i$ and all $\nabla G_i$ are continuously differentiable, there exists a (common) Lipschitz constant $L$ on $\mathcal{K}$. For $i \notin I_\epsilon(x)$ we have
	\begin{equation*}
		\begin{aligned}
			|G_i(\pi(x+tv)) - G_i(x)| &\leq L \| \pi(x+tv) - x \| = t L \Big\| \frac{\pi(x+tv) - x}{t} - v + v \Big\| \\
			&\leq t L \left( \Big\| \frac{\pi(x+tv) - x}{t} - v \Big\| + \| v \| \right) \leq t L (\delta + \| v \|) 
		\end{aligned}
	\end{equation*}
	for $t \leq T$ with some $T > 0$, where the latter estimation is similar to the proof of Lemma~\ref{lem:stepSizeS1Bound1}. Since
	\begin{equation*}
		t L (\delta + \| v \|) \leq \epsilon \Leftrightarrow t \leq \frac{\epsilon}{L (\delta + \| v \|)},
	\end{equation*}
	we have
	\begin{equation*}
		|G_i(\pi(x+tv)) - G_i(x)| \leq \epsilon \quad \forall t \in \left[ 0, \min\left( T, \frac{\epsilon}{L (\delta + \| v \|)} \right) \right],
	\end{equation*}
	which (due to $G_i(x) < -\epsilon$) results in
	\begin{equation*}
		G_i(\pi(x+tv)) \leq 0 \quad \forall t \in \left[ 0, \min\left( T, \frac{\epsilon}{L (\delta + \| v \|)} \right) \right].
	\end{equation*}
	For $i \in I_\epsilon(x)$ we apply Lemma \ref{lem:lipIneq} to get
	\begin{equation}
		\begin{aligned}
			& G_i(\pi(x+tv)) \leq G_i(x) + \nabla G_i(x) (\pi(x+tv) - x) + \frac{1}{2} L \| \pi(x+tv) - x \|^2 \\
			&\leq t \nabla G_i(x) \left( \left( \frac{\pi(x+tv) - x}{t} - v \right) + v \right) + \frac{1}{2} L t^2 \Big\| \left( \frac{\pi(x+tv) - x}{t} - v \right) + v  \Big\|^2. \label{eq:LemStepSizeS1Bound2Eq1}
		\end{aligned}
	\end{equation}
	Since all $\nabla G_i$ are bounded on $\mathcal{K}$, the first term can be estimated by
	\begin{equation*}
		\begin{aligned}
			&t \nabla G_i(x) \left( \left( \frac{\pi(x+tv) - x}{t} - v \right) + v \right)  \\
			&= t \nabla G_i(x) \left( \frac{\pi(x+tv) - x}{t} - v \right) + t \nabla G_i(x) v \\ 
			&\leq t (\delta + \nabla G_i(x) v) \leq t(\delta + \rho)
		\end{aligned}
	\end{equation*}
	for $t < T$ with some $T > 0$ (again like in the proof of Lemma \ref{lem:stepSizeS1Bound1}). For the second term we have
	\begin{equation*}
		\begin{aligned}
			&\frac{1}{2} L t^2 \Big\| \left( \frac{\pi(x+tv) - x}{t} - v \right) + v  \Big\|^2 \leq \frac{1}{2} L t^2 \left( \Big\| \frac{\pi(x+tv) - x}{t} - v \Big\| + \| v \| \right)^2 \\
			&\leq \frac{1}{2} L t^2 (\delta + \| v \|)^2
		\end{aligned}
	\end{equation*}
	for $t < T$. Combining both estimates with \eqref{eq:LemStepSizeS1Bound2Eq1}, we obtain
	\begin{equation*}
		G_i(\pi(x+tv)) \leq t(\delta + \rho) + \frac{1}{2} L t^2 (\delta + \| v \|)^2.
	\end{equation*}
	Therefore, we have $G_i(\pi(x+tv)) \leq 0$ if
	\begin{equation*}
		\begin{aligned}
			&t(\delta + \rho) + \frac{1}{2} L t^2 (\delta + \| v \|)^2 \leq 0 \\
			\Leftrightarrow \quad &t \leq -\frac{2 (\rho + \delta)}{L (\delta + \| v \|)^2}.
		\end{aligned}
	\end{equation*}
	Combining the bounds for $i \in I_\epsilon(x)$ and $i \notin I_\epsilon(x)$ (i.e.~taking the minimum of all upper bounds for $t$) completes the proof. 
\end{proof}

As mentioned before, a drawback of using an active set strategy is the fact that this approach naturally causes discontinuities for the descent direction at the boundary of our feasible set. But fortunately, $\alpha_1(\cdot,\epsilon)$ is still upper semi-continuous, which is shown in the next lemma. We will later see that this is sufficient to proof convergence.

\begin{lemma} \label{lem:alphaSemiCont}
	Let $(x_l)_l$ be a sequence in $\mathcal{N}$ with $\lim\limits_{l \rightarrow \infty} x_l = \bar{x}$. Then
	\begin{equation*}
		\limsup\limits_{l \rightarrow \infty} \alpha_1(x_l,\epsilon) \leq \alpha_1(\bar{x},\epsilon).
	\end{equation*}
\end{lemma}
\begin{proof}
	Let $I'$ be the set of indices of the ICs which are active for infinitely many elements of $(x_l)_l$, i.e.
	\begin{equation*}
		I' := \bigcap_{k \in \mathbb{N}} \bigcup_{l \geq k} I_\epsilon(x_l).
	\end{equation*}
	We first show that $I' \subseteq I_\epsilon(\bar{x})$. If $I' = \emptyset$ we obviously have $I' \subseteq I_\epsilon(\bar{x})$. Therefore, let $I' \neq \emptyset$ and $j \in I'$. Then there has to be a subsequence $(x_{l_u})_u$ of $(x_l)_l$ such that $j \in I_\epsilon(x_{l_u})$ for all $u \in \mathbb{N}$. Thus,
	\begin{equation*}
		G_j(x_{l_u}) \geq -\epsilon \quad \forall u \in \mathbb{N}.
	\end{equation*}
	Since $(x_{l_u})_u$ converges to $\bar{x}$ and by the continuity of $G_j$, we also have 
	\begin{equation*}
		G_j(\bar{x}) \geq -\epsilon,
	\end{equation*}
	hence $j \in I_\epsilon(\bar{x})$ and consequently, $I' \subseteq I_\epsilon(\bar{x})$.\\
	Let $(x_{l_k})_k$ be a subsequence of $(x_l)_l$ with
	\begin{equation*}
		\lim\limits_{k \rightarrow \infty} \alpha_1(x_{l_k},\epsilon) = \limsup\limits_{l \rightarrow \infty} \alpha_1(x_l,\epsilon).
	\end{equation*}
	Since there are only finitely many ICs, we only have finitely many possible active sets, and there has to be some $I^s \subseteq \{1,...,m_G\}$ which occurs infinitely many times in $(I_\epsilon(x_{l_k}))_k$. W.l.o.g.~assume that $I^s = I_\epsilon(x_{l_k})$ holds for all $k \in \mathbb{N}$. By Lemma \ref{lem:avCont}, the map $\alpha_1(\cdot,\epsilon)$ is continuous on the closed set $\{x \in \mathcal{N}: I_\epsilon(x) = I^s \}$ (by extending the objective function $F$ with the ICs in $I^s$). Thus, $\lim\limits_{k \rightarrow \infty} \alpha_1(x_{l_k},\epsilon)$ is the optimal value of a subproblem like \eqref{SP1} at $\bar{x}$, except we take $I^s$ as the active set. This completes the proof since $\alpha_1(\bar{x},\epsilon)$ is the optimal value of \eqref{SP1} at $\bar{x}$ with the actual active set $I_\epsilon(\bar{x})$ at $\bar{x}$ and $I^s \subseteq I' \subseteq I_\epsilon(\bar{x})$ holds. (The optimal value can only get larger when additional conditions in $I_\epsilon(\bar{x}) \setminus I^s$ are considered.)
\end{proof}

Lemma \ref{lem:stepSizeS1Bound1}, \ref{lem:stepSizeS1Bound2} and \ref{lem:alphaSemiCont} now enable us to prove the following convergence result.

\begin{theorem} \label{thm:convS1}
	Let $(x_l)_l$ be a sequence generated by Algorithm \ref{algo:S1} with $\epsilon > 0$. Then 
	\begin{equation*}
		\alpha_1(\bar{x},\epsilon) = 0
	\end{equation*}
	for all accumulation points $\bar{x}$ of $(x_l)_l$ or, if $(x_l)_l$ is finite, the last element $\bar{x}$ of the sequence.
\end{theorem}
\begin{proof}
	The case where $(x_l)_l$ is finite is obvious (cf.~Step 5 in Algorithm~\ref{algo:S1}), so assume that $(x_l)_l$ is infinite. Let $\bar{x}$ be an accumulation point of $(x_l)_l$. Since each component of $(F(x_l))_l$ is monotonically decreasing and $F$ is continuous, $(F(x_l))_l$ has to converge, i.e.
	\begin{equation*}
		\lim\limits_{l \rightarrow \infty} \left( F(x_{l+1}) - F(x_l) \right) = 0.
	\end{equation*}
	We therefore have
	\begin{equation}
		\lim\limits_{l \rightarrow \infty} t_l DF(x_l) v_l = 0. \label{eq:thmConvS1Eq1}
	\end{equation}
	Let $(x_{l_u})_u$ be a subsequence of $(x_l)_l$ with $\lim\limits_{u \rightarrow \infty} x_{l_u} = \bar{x}$. \\
	We now show the desired result by contradiction. Assume $\alpha_1(\bar{x},\epsilon) < 0$. By Lemma \ref{lem:alphaSemiCont} we have  
	\begin{equation*}
		\limsup\limits_{u \rightarrow \infty} \alpha_1(x_{l_u},\epsilon) < 0.
	\end{equation*}
	Let $\rho < 0$ so that $\alpha_1(x_{l_u},\epsilon) \leq \rho$ for all $u \in \mathbb{N}$. By definition of $\alpha_1$ we therefore have
	\begin{equation}
		\nabla F_i(x_{l_u}) v_{l_u} \leq \rho \quad \textrm{and} \quad \nabla G_j(x_{l_u}) v_{l_u} \leq \rho \label{eq:thmConvS1Eq2}
	\end{equation}
	for all $u \in \mathbb{N}$, $i \in \{1,...,m\}$, $j \in I_\epsilon(x_{l_u})$. Since $(\| v_{l_u} \|)_u$ is bounded, there is some $C_v$ with $\| v_{l_u} \| \leq C_v$ for all $u \in \mathbb{N}$. Due to the convergence of $(x_{l_u})_u$, all elements of the sequence are contained in a compact set. By Lemma \ref{lem:stepSizeS1Bound1}, all step lengths in
	\begin{equation*}
		\left[ 0, \min \left( T, -\frac{2 (1 - \sigma) \rho + \delta}{L^3 C_v^2} \right) \right]
	\end{equation*}
	satisfy the Armijo condition for arbitrary $\delta > 0$ and proper $T > 0$. With $\delta = -(1 - \sigma) \rho$ and proper $T_1 > 0$, all step lengths in
	\begin{equation*}
		\left[ 0, \min \left( T_1, -\frac{(1 - \sigma) \rho}{L^3 C_v^2} \right) \right]
	\end{equation*}
	satisfy the Armijo condition. By Lemma \ref{lem:stepSizeS1Bound2} we have $x_{l_u} + t v_{l_u} \in \mathcal{N}$ for all step lengths $t$ in
	\begin{equation*}
		\left[0, \min\left(T, \frac{\epsilon}{L (C_v + \delta)} , -\frac{2 (\rho + \delta)}{L(C_v + \delta)^2} \right)\right]
	\end{equation*}
	for arbitrary $\delta > 0$ and proper $T > 0$. With $\delta = -\frac{1}{2} \rho$ and a properly chosen $T_2 > 0$, the last interval becomes
	\begin{equation*}
		\left[0, \min\left(T_2, \frac{\epsilon}{L (C_v - \frac{1}{2} \rho)} , -\frac{\rho}{L(C_v -\frac{1}{2} \rho)^2} \right)\right].
	\end{equation*}
	We now define 
	\begin{equation*}
		t' := \min\left( T_1, -\frac{(1 - \sigma) \rho}{L^3 C_v^2}, T_2, \frac{\epsilon}{L (C_v - \frac{1}{2} \rho)} , -\frac{\rho}{L(C_v -\frac{1}{2} \rho)^2} \right).
	\end{equation*}
	Observe that $t'$ does not depend on the index $u$ of $(x_{l_u})_u$ and since $\epsilon > 0$, we have $t' > 0$. Since all $t \in [0,t']$ satisfy the Armijo condition and $x_{l_u} + t v_{l_u} \in \mathcal{N}$ for all $u$, we know that $t_{l_u}$ has a lower bound. \\
	By equality \eqref{eq:thmConvS1Eq1} we therefore have
	\begin{equation*}
		\lim\limits_{l \rightarrow \infty} DF(x_l) v_l = 0,
	\end{equation*}
	which contradicts \eqref{eq:thmConvS1Eq2}, and we have $\alpha_1(\bar{x},\epsilon) = 0$.
\end{proof}

We conclude our results about Strategy 1 with two remarks.

\begin{remark}
	Unfortunately, it is essential for the proof of the last theorem to choose $\epsilon > 0$, so we can not expect the accumulation points to be Pareto critical (cf. Remark \ref{rem:propertiesAlphaS1}). But since 
	\begin{equation*}
		\alpha_1(x,0) = 0 \quad \Rightarrow \quad \alpha_1(x,\epsilon) = 0 \quad \forall \epsilon  > 0,
	\end{equation*}
	Theorem \ref{thm:convS1} still shows that accumulation points satisfy a necessary optimality condition, it is just not as strict as Pareto criticality. To obtain convergence, one could decrease $\epsilon$ more and more during execution of Algorithm \ref{algo:S1}. For the unconstrained case, this was done in \cite[Algorithm 2]{FS2000}, and indeed results in  Pareto criticality of accumulation points. This result indicates that the accumulation points of our algorithm are close to Pareto critical points if we choose small values for $\epsilon$.
\end{remark}

\begin{remark} \label{rem:modConvS1}
	Observe that similar to the proof of Theorem \ref{thm:convEq}, the proof of Theorem \ref{thm:convS1} also works for slightly more general sequences. (This will be used in a later proof and holds mainly due to the fact that we consider subsequences in the proof Theorem \ref{thm:convS1}.)\\
	Let $(x_l)_l \in \mathcal{M}$, $K \subseteq \mathbb{N}$ with $| K | = \infty$ so that 
	\begin{itemize}
		\item $(F_i(x_l))_l$ is monotonically decreasing for all $i \in \{1,...,m\}$ and
		\item the step from $x_k$ to $x_{k+1}$ was realized by an iteration of Algorithm \ref{algo:S1} for all $k \in K$.
	\end{itemize}
	Let $\bar{x} \in \mathcal{M}$ such that there exists a sequence of indices $(l_s)_s \in K$ with $\lim_{s \rightarrow \infty} x_{l_s} = \bar{x}$. Then $\alpha_1(\bar{x},\epsilon) = 0$. 
\end{remark}

\subsubsection{Strategy 2 -- Active inequalities as equalities}
\label{subsubsec:Descent_EC_IC_S2}
We will now introduce the second active set strategy, where active inequalities are considered as equality constraints when calculating the descent direction. To be able to do so, we have to impose an additional assumption on the ICs.
\begin{itemize}
	\item[] \textbf{(A5):} The elements in  
	\begin{equation*}
		\{ \nabla H_1(x),..., \nabla H_{m_H}(x) \} \cup \{\nabla G_i(x) : i \in I_0(x)\}
	\end{equation*}
	are linearly independent for all $x \in \mathcal{N}$.
\end{itemize}
This extends (A2) and ensures that the set
\begin{equation*}
	\mathcal{M}_I := \mathcal{M} \cap \bigcap_{i \in I} G_i^{-1}(\{0\})
\end{equation*}
is either empty or a $C^2$-submanifold of $\R^n$. For $\mathcal{M}_I \neq \emptyset$ define
\begin{equation*}
	\pi_I: \R^n \rightarrow \mathcal{M}_I 
\end{equation*}
as the ``projection'' onto $\mathcal{M}_I$ (cf. Definition \ref{def:projection}). Treating the active ICs as ECs in \eqref{SPe}, we obtain the following subproblem for a given $x \in \mathcal{N}$.
\begin{equation*} \label{SP2}
	\begin{aligned}
		&\alpha_2(x) := && \underset{(v,\beta) \in \R^{n+1}}{\text{min}} &&  \beta + \frac{1}{2} \| v \|^2  && \\
		& && \text{s.t.} &&  \nabla F_i(x) v \leq \beta &&\forall i \in \{1,...,m\}, \\
		& && && \nabla H_j(x) v = 0 &&\forall j \in \{1,...,m_H\}, \\
		& && && \nabla G_l(x) v = 0 &&\forall l \in I_0(x).
	\end{aligned} \tag{SP2}
\end{equation*}

The following properties of \eqref{SPe} can be transferred to \eqref{SP2}.

\begin{remark} \label{rem:propertiesAlphaS2} 
	\begin{enumerate}
		\item By Lemma \ref{lem:critical_alpha0}, $x$ is Pareto critical in $\mathcal{M}_{I_0(x)}$ iff $\alpha_2(x) = 0$.
		\item By Lemma \ref{lem:alphaUnique_convexComb} (with extended ECs) we obtain uniqueness of the solution of \eqref{SP2}. 
	\end{enumerate}
\end{remark}

Unfortunately, $\alpha_2$ can not be used as a criterion to test for Pareto criticality in $\mathcal{N}$. $\alpha_2(x) = 0$ only means that the cone of descent directions of $F$, the tangent space of the ECs and the tangent space of the active inequalities have no intersection. On the one hand, this occurs when the cone of descent directions points outside the feasible set which means that $x$ is indeed Pareto critical. On the other hand, this also occurs when the cone points inside the feasible set. In this case $x$, is not Pareto critical. Consequently, $\alpha_2$ can only be used to test for Pareto criticality with respect to $\mathcal{M}_{I_0(x)}$, i.e.~the constrained MOP where the active ICs are actually ECs (and there are no other ICs). The following lemma shows a simple relation between Pareto criticality in $\mathcal{N}$ and in $\mathcal{M}_{I_0(x)}$.

\begin{lemma}
	A Pareto critical point $x$ in $\mathcal{N}$ is Pareto critical in $\mathcal{M}_{I_0(x)}$.
\end{lemma}
\begin{proof}
	By definition, there is no $v \in \text{ker}(DH(x))$ such that
	\begin{equation*}
		DF(x) v < 0 \quad \textrm{and} \quad \nabla G_i(x) \leq 0 \quad \forall i \in I_0(x).
	\end{equation*}
	For $w \in \R^n$ let $w^\perp$ denote the orthogonal complement of the linear subspace of $\R^n$ spanned by $w$. In particular, there exists no 
	\begin{equation*}
		v \in \text{ker}(DH(x)) \cap \bigcap_{i \in I_0(x)} \nabla G_i(x)^{\perp} = T_x (\mathcal{M}_{I_0(x)})
	\end{equation*}
	such that $DF(x) v < 0$. Therefore, $x$ is Pareto critical in $\mathcal{M}_{I_0(x)}$.
\end{proof}

The fact that $\alpha_2$ can not be used to test for Pareto criticality (in $\mathcal{N}$) poses a problem when we want to use \eqref{SP2} to calculate a descent direction. For a general sequence $(x_n)_n$, the active set $I_0(x)$ can change in each iteration. Consequently, active ICs can also become inactive. However, using \eqref{SP2}, ICs can not become inactive. By Remark \ref{rem:propertiesAlphaS2} and Theorem \ref{thm:convEq}, simply using Algorithm~\ref{algo:eq} (with \eqref{SP2} instead of \eqref{SPe})
will result in a point $x \in \mathcal{N}$ that is Pareto critical in $\mathcal{M}_I$ for some $I \subseteq \{1,...,m_G\}$, but not necessarily Pareto critical in $\mathcal{N}$. This means we can not just use \eqref{SP2} on its own. To solve this problem, we need a mechanism that deactivates active inequalities when appropriate. To this end, we combine Algorithms \ref{algo:eq} and \ref{algo:S1} and introduce a parameter $\eta > 0$. If in the current iteration we have $\alpha_2 > -\eta$ after solving \eqref{SP2}, we calculate a descent direction using \eqref{SP1} and use the step length in Algorithm \ref{algo:S1}. Otherwise, we take a step in the direction we calculated with \eqref{SP2}. The entire procedure is summarized in Algorithm \ref{algo:S2}.

\begin{algorithm} 
	\caption{(Descent method for equality and inequality constrained MOPs, Strategy 2)}
	\label{algo:S2}
	\begin{algorithmic}[1]
		\Require $x \in \R^n$, $\beta_0 > 0$, $\beta \in (0,1)$, $\epsilon > 0$, $\eta > 0$, $\sigma \in (0,1)$.
		\State Compute some $x_0$ with $\| x - x_0 \| = \min\{\| x - x_0 \| : x_0 \in \mathcal{N \}}$.
		\For{$l = 0$, $1$, ...}
			\State Identify the active set $I_0(x_l)$.
			\State Compute the solution $v_l$ of \eqref{SP2} at $x_l$.
			\If{$\alpha_2(x_l) > -\eta$}
				\State Compute $t_l$ and $v_l$ as in Algorithm \ref{algo:S1}. If $\alpha_1(x_l,\epsilon) = 0$, stop.
				\State Set $x_{l+1} = \pi(x_l + t_l v_l)$.
			\Else
				\State Compute 
					\begin{equation*}
						k_l := \min\{k \in \mathbb{N}: F(\pi_{I_0(x_l)}(x_l + \beta_0 \beta^k v_l)) < F(x) + \sigma \beta_0 \beta^k DF(x_l) v_l \}.
					\end{equation*}
				\Statex \hspace{23pt} and set $t_l := \beta_0 \beta^{k_l}$.
					\If{$\pi(x_l + t_l v_l) \notin \mathcal{N}$}
						\State Compute some $t < t_l$ such that 
							\begin{equation*}
								\pi_{I_0(x_l)}(x_l + t v_l) \in \mathcal{N} \quad \textrm{and} \quad | I_0(\pi_{I_0(x_l)}(x_l + t v_l)) | > | I_0(x_l) |
							\end{equation*}
						\Statex \hspace{37pt} and the Armijo condition holds. Set $t_l = t$.
					\EndIf
				\State Set $x_{l+1} = \pi_{I_0(x_l)}(x_l + t_l v_l)$.
			\EndIf
		\EndFor
	\end{algorithmic}
\end{algorithm}


Since Algorithm \ref{algo:S2} generates the exact same sequence as Algorithm \ref{algo:S1} if we take $\eta$ large enough, it can be seen as a generalization of Algorithm \ref{algo:S1}. To obtain the desired behavior of ``moving along the boundary'', one has to consider two points when choosing $\eta$: On the one hand, $\eta$ should be small enough such that active boundaries that indeed possess a Pareto critical point are not activated and deactivated too often. On the other hand, $\eta$ should be large enough so that the sequence actually moves along the boundary and does not ``bounce off'' too early. Additionally, if it is known that all Pareto critical points in $\mathcal{M}_{I_0(x)}$ are also Pareto critical in $\mathcal{N}$, $\eta < 0$ can be chosen such that Strategy 1 is never used during execution of Algorithm \ref{algo:S2}.

In order to show that Algorithm \ref{algo:S2} is well defined, we have to prove existence of the step length.
\begin{lemma}
	Let $x \in \mathcal{N}$, $\beta_0 > 0$, $\beta \in (0,1)$, $\sigma \in (0,1)$ and 
	\begin{equation*}
		v \in \text{ker}(DH(x)) \cap \bigcap_{i \in I_0(x)} \nabla G_i(x)^{\perp} \quad \textrm{with} \quad DF(x)v < 0.
	\end{equation*}
	Then there is some $l \in \mathbb{N}$ so that
	\begin{equation*}
		F(\pi_{I_0(x)}(x + \beta_0 \beta^k v)) < F(x) + \sigma \beta_0 \beta^k DF(x)v \quad \forall k > l
	\end{equation*}
	and
	\begin{equation*}
		\pi_{I_0(x)}(x + \beta_0 \beta^k v) \in \mathcal{N} \quad \forall k > l. 
	\end{equation*}
\end{lemma}
\begin{proof}
	Using Lemma \ref{lem:stepSizeS1Exists} with the active ICs at $I_0(x)$ as additional ECs and $\epsilon = 0$.
\end{proof}

\begin{remark}
	In theory, it is still possible that Step 11 in Algorithm \ref{algo:S2} fails to find a step length that satisfies the conditions, because the Armijo condition does not have to hold for all $t < t_l$. In that case, this problem can be avoided by choosing smaller values for $\beta$ and $\beta_0$.
\end{remark}

Since Algorithm \ref{algo:S2} is well defined, we know that it generates a sequence $(x_l)_l$ in $\mathcal{N}$ with $F(x_{l+1}) < F(x_l)$ for all $l \geq 0$. The following theorem shows that sequences generated by Algorithm \ref{algo:S2} converge in the same way as sequences generated by Algorithm \ref{algo:S1}. Since Algorithm \ref{algo:S2} can be thought of as a combination of Algorithm \ref{algo:eq} with modified ECs and Algorithm \ref{algo:S1}, the idea of the proof is a case analysis of how a generated sequence is influenced by those algorithms. We call Step 5 and Step 10 in Algorithm \ref{algo:S2} \emph{active in iteration $i$}, if the respective conditions of those steps are met in the $i$-th iteration.

\begin{theorem}
	Let $(x_l)_l$ be a sequence generated by Algorithm \ref{algo:S2} with $\epsilon > 0$. Then 
	\begin{equation*}
		\alpha_1(\bar{x},\epsilon) = 0
	\end{equation*}
	for all accumulation points $\bar{x}$ of $(x_l)_l$ or, if $(x_l)_l$ is finite, the last element $\bar{x}$ of the sequence.
\end{theorem}
\begin{proof}
	If $(x_l)_l$ is finite, the stopping criterion in Step 6 was met, so we are done. Let $(x_l)_l$ be infinite, $\bar{x}$ an accumulation point and $(x_{l_s})_s$ a subsequence of $(x_l)_l$ so that $\lim\limits_{s \rightarrow \infty} x_{l_s} = \bar{x}$. Consider the following four cases for the behavior of Algorithm \ref{algo:S2}: \\
	\emph{Case 1:} In the iterations where $l = l_s$ with $s \in \mathbb{N}$, Step 5 is active infinitely many times. Then the proof follows from Theorem \ref{thm:convS1} and Remark \ref{rem:modConvS1}. \\
	\emph{Case 2:} In the iterations where $l = l_s$ with $s \in \mathbb{N}$, Step 5 and 10 are both only active a finite number of times. W.l.o.g. they are never active. Since the power set $\mathcal{P}(\{ 1,...,m_G \})$ is finite, there has to be some $I \in \mathcal{P}(\{ 1,...,m_G \})$ and a subsequence $(x_{l_u})_u$ of $(x_{l_s})_s$ so that $I_0(x_{l_u}) = I$ for all $u \in \mathbb{N}$. In this case, the proof of Theorem \ref{thm:convEq} and Remark \ref{rem:modConvEq} (with modified ECs) shows that $\lim\limits_{u \rightarrow \infty} \alpha_2(x_{l_u}) = 0$, so Step 5 has to be active for some iteration with $l = k_u$, which is a contradiction. Thus Case 2 can not occur. \\
	\emph{Case 3:} In the iterations where $l = l_s$ with $s \in \mathbb{N}$, Step 5 is only active a finite number of times and $\limsup\limits_{s \rightarrow \infty} t_{l_s} > 0$. W.l.o.g. Step 5 is never active and $\lim\limits_{s \rightarrow \infty} t_{l_s} > 0$. As in Case 2 there has to be a subsequence $(x_{l_u})_u$ of $(x_{l_s})_s$ and some $I$ so that $I_0(x_{l_u}) = I$ for all $u \in \mathbb{N}$. However, the first case in the proof of Theorem \ref{thm:convEq} (with modified ECs) yields  $\lim\limits_{u \rightarrow \infty} \alpha_2(x_{l_u}) = 0$, so Step 5 has to be active at some point which is a contradiction. Consequently, Case 3 can not occur either. \\
	\emph{Case 4:} In the iterations where $l = l_s$ with $s \in \mathbb{N}$, Step 5 is only active a finite amount of times (w.l.o.g. never), Step 10 is active an infinite amount of times (w.l.o.g. always) and $\lim\limits_{s \rightarrow \infty} t_{l_s} = 0$. We have
	\begin{equation*}
		\| \bar{x} - x_{l_s + 1} \| = \| \bar{x} - x_{l_s} + x_{l_s} - x_{l_s + 1} \| \leq \| \bar{x} - x_{l_s} \| + \| x_{l_s} - x_{l_s + 1} \|.
	\end{equation*}
	For $s \in \mathbb{N}$ the first term gets arbitrarily small since we have $\lim_{s \rightarrow \infty} x_{l_s} = \bar{x}$ by assumption. As in Case 2, we can assume w.l.o.g.~that $I_0(x_{l_s}) = I$. Then the second term becomes arbitrarily small since $\lim\limits_{s \rightarrow \infty} t_{l_s} = 0$, $(\| v_{l_s} \|)_s$ is bounded and the projection $\pi_{I_0(x_{l_s})} = \pi_I$ is continuous. Thus $(x_{l_s + 1})_s$ is another sequence that converges to $\bar{x}$. Since Step 10 is always active for $l = l_s$, we have 
	\begin{equation}
		I_0(x_{l_s + 1}) \geq I_0(x_{l_s}) + 1 \quad \forall s \in \mathbb{N}. \label{eq:thmConvS2Eq1}
	\end{equation}
	If the prerequisites of Case 1 hold for $(x_{l_s + 1})_s$, the proof is complete. So w.l.o.g.~assume that $(x_{l_s + 1})_s$ satisfies the prerequisites of Case 4. Now consider the sequence $(x_{l_s + 2})_s$. Using the same argument as above, we only have to consider Case 4 for this sequence and we obtain
	\begin{equation*}
		I_0(x_{l_s + 2}) \geq I_0(x_{l_s}) + 2 \quad \textrm{for infinitely many } s \in \mathbb{N}.
	\end{equation*}
	(Note that this inequality does not have to hold for all $s \in \mathbb{N}$ since it is possible that we only consider a subsequence in Case 4.) If we continue this procedure, there has to be some $k \in \mathbb{N}$ such that the sequence $(x_{l_s + k})_s$ converges to $\bar{x}$ but does not satisfy the prerequisites of Case 4, since there is only a finite amount of ICs and by inequality \eqref{eq:thmConvS2Eq1} the amount of active ICs increases in each step of our procedure. Thus, $(x_{l_s + k})_s$ has to satisfy the prerequisites of Case 1, which completes the proof.
\end{proof}

\section{Numerical results}
\label{sec:Numerical_Results}
In this section we will present and discuss the typical behavior of our method using an academic example. Since Algorithm \ref{algo:S1} is a special case of Algorithm \ref{algo:S2} (when choosing $\eta$ large enough), we will from now on only consider Algorithm \ref{algo:S2} with varying $\eta$. Consider the following example for an inequality constrained MOP.

\begin{example} \label{ex:1}
	\begin{figure}[ht]
		\centering
		\includegraphics[scale=0.7]{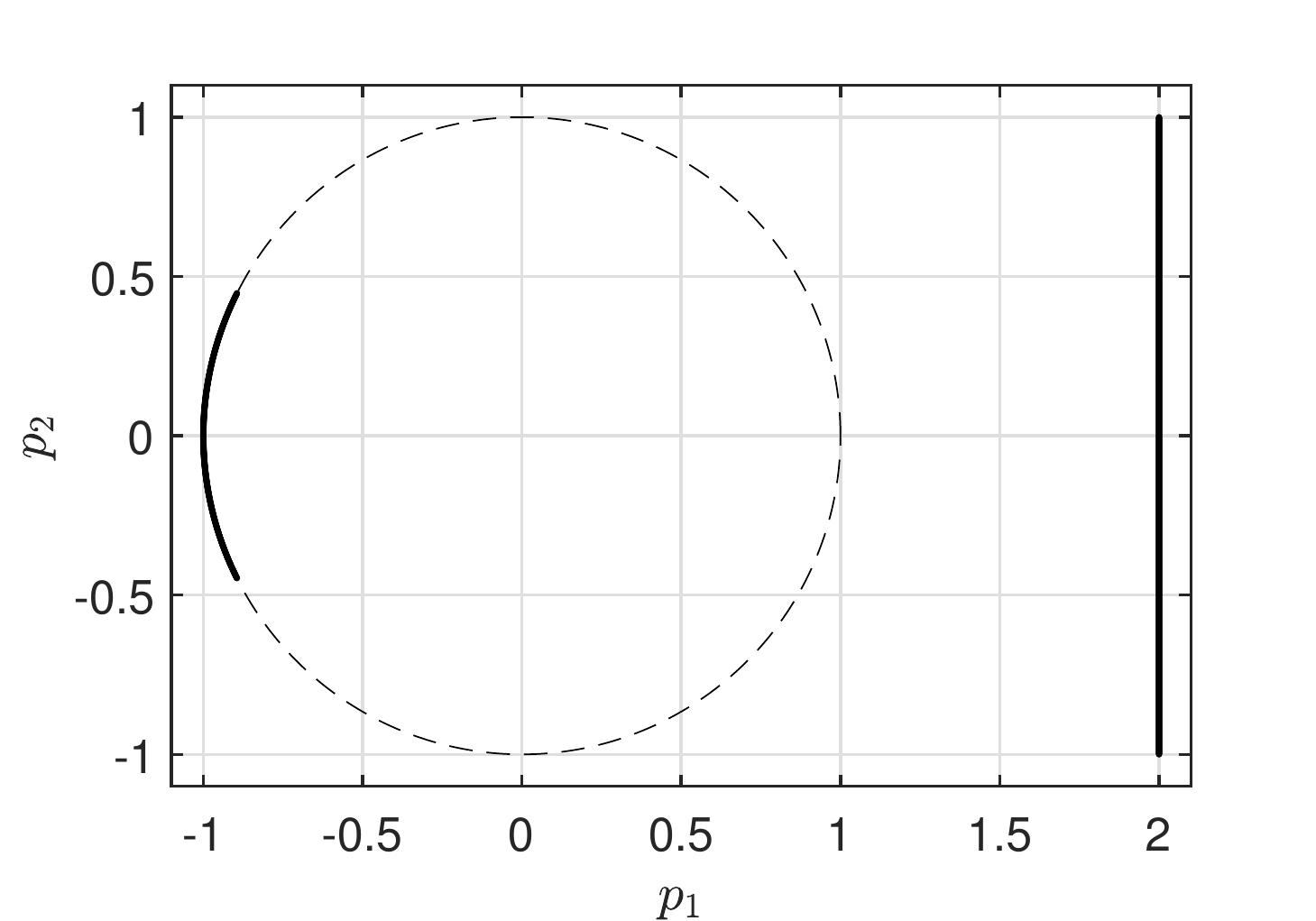}
		\caption{The unit circle (dashed) and the set of Pareto critical points (thick).}
		\label{fig:ex1Critical}
	\end{figure}
	Let 
	\begin{equation*}
		F : \R^2 \rightarrow \R^2, \quad
		\begin{pmatrix}
			x_1\\
			x_2
		\end{pmatrix}
		\mapsto 
		\begin{pmatrix}
			(x_1 - 2)^2 + (x_2 - 1)^2\\
			(x_1 - 2)^2 + (x_2 + 1)^2
		\end{pmatrix}
	\end{equation*}
	and
	\begin{equation*}
		G : \R^2 \mapsto \R, \quad
		\begin{pmatrix}
			x_1\\
			x_2
		\end{pmatrix}
		\mapsto -x_1^2 - x_2^2 + 1.
	\end{equation*}
	Consider the MOP
	\begin{equation*}
		\begin{aligned}
			& \underset{x \in \R^{2}}{\text{min}} && F(x), && \\
			& \text{s.t.} &&  G(x) \leq 0. &&
		\end{aligned}
	\end{equation*}
	The set of feasible points is $\R^2$ without the interior of the unit circle (cf.~the dashed line in Figure~\ref{fig:ex1Critical}). The set of Pareto critical points is
	\begin{equation*}
		\left\{ 
		\begin{pmatrix}
			\cos(t)\\
			\sin(t)
		\end{pmatrix}
		: t \in \left[ \pi-\theta, \pi + \theta \right]
		\right\} 
		\cup 
		\left\{ 
		\begin{pmatrix}
			2\\
			-1+s
		\end{pmatrix}
		: s \in [0,2] \right\},
	\end{equation*}
	with $\theta = \arctan(\frac{1}{2})$ and the line between $(2,1)^T$ and $(2,-1)^T$ being globally Pareto optimal. Figure \ref{fig:ex1Critical} shows the set of Pareto critical points.
\end{example}

\begin{figure}[t]
	\centering
	\includegraphics[scale=0.55]{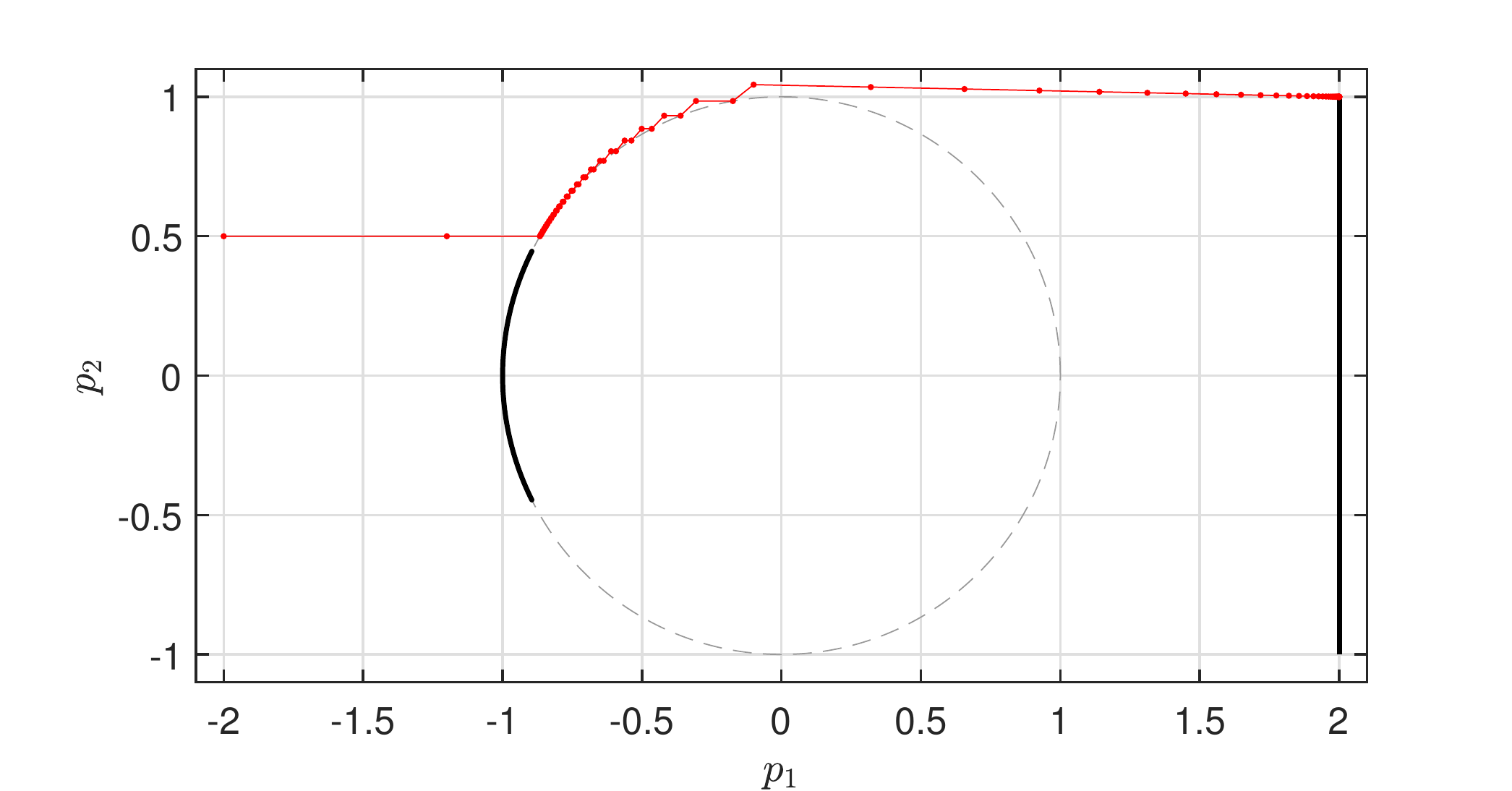}
	\caption{Algorithm \ref{algo:S2} with $\eta = \infty$ starting in $(-2,0.5)^T$ for Example \ref{ex:1}.}
	\label{fig:ex1EtaInf}
\end{figure}
As parameters for Algorithm \ref{algo:S2} we choose
\begin{equation*}
	\beta = \frac{1}{2},~\beta_0 = \frac{1}{10},~\epsilon = 10^{-4}
\end{equation*}
and $\eta \in \{1,\infty\}$. (The parameters for the Armijo step length are set to relatively small values for better visibility of the behavior of the algorithm.) Figure \ref{fig:ex1EtaInf} shows a sequence generated by Algorithm \ref{algo:S2} with $\eta = \infty$ (i.e.~Strategy~1 (Algorithm \ref{algo:S1})). When the sequence hits the boundary, the IC $G$ is activated and in the following subproblem, $G$ is treated as an additional objective function. Minimizing the active ICs results in leaving the boundary.
During the first few steps on the boundary, all possible descent directions are almost orthogonal to the gradients $\nabla F_1$ and $\nabla F_2$. Consequently, the Armijo step lengths are very short. The effect of this is that the sequence hits the boundary many times before the descent directions allow for more significant steps. (Note that this behavior is amplified by our choice of parameters for the Armijo step length.) After the sequence has passed the boundary, the behavior is exactly as in the unconstrained steepest descent method.
\begin{figure}[hb!]
	\centering
	\includegraphics[scale=0.55]{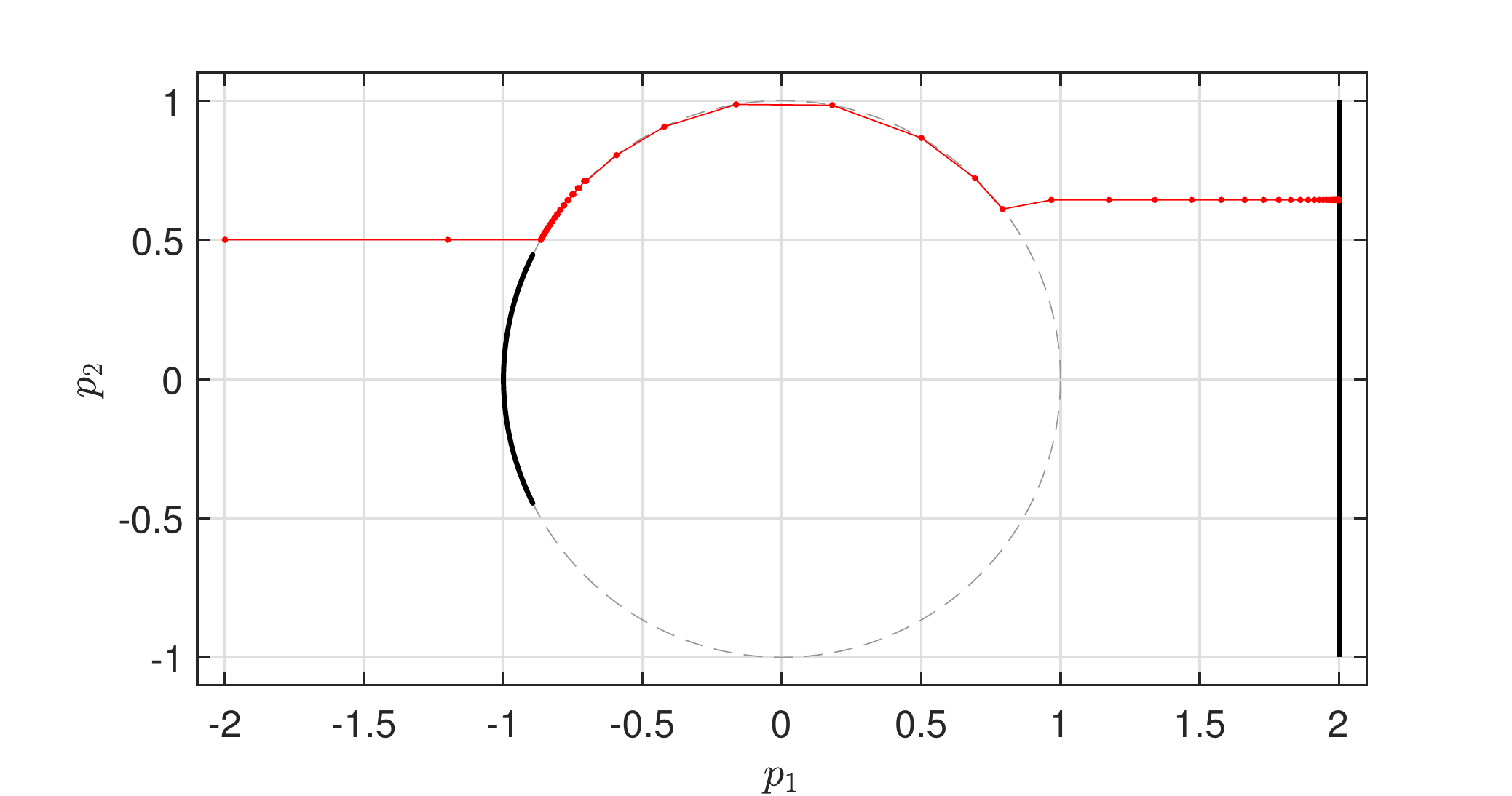}
	\caption{Algorithm \ref{algo:S2} with $\eta = 1$ starting in $(-2,0.5)^T$ for Example \ref{ex:1}.}
	\label{fig:ex1Eta1}
\end{figure}

Figure \ref{fig:ex1Eta1} shows a sequence generated by Algorithm \ref{algo:S2} with $\eta = 1$ which means that active ICs are often treated as ECs instead of objectives. When the sequence first hits the boundary, the IC becomes active and in the following iteration, a solution of \eqref{SP2} is computed. Since the descent direction has to lie in the tangent space of $G^{-1}(\{0\})$ at the current boundary point and there is no descent direction of acceptable quality in this subspace, the optimal value of \eqref{SP2} is less than $-\eta$. Consequently, the solution of \eqref{SP2} is discarded and \eqref{SP1} is solved to compute a descent direction. This means that the sequence tries to leave the boundary and we obtain the same behavior as in Figure \ref{fig:ex1EtaInf}. This occurs as long as the optimal value of \eqref{SP2} on the boundary is less than $-\eta$. If it is greater than $-\eta$ (i.e.~close to zero), the sequence ``sticks'' to the boundary as our method treats the active inequality as an equality. This behavior is observed as long as the optimal value of \eqref{SP2} is larger than $-\eta$. Then the algorithm will again use \eqref{SP1} which causes the sequence to leave the boundary and behave like the unconstrained steepest descent method from then on. 

The above example shows that it makes sense to use Algorithm \ref{algo:S2} with a finite $\eta$ since a lot less steps are required on the boundary (and in general) than with $\eta = \infty$. The fact that the behavior with $\eta = 1$ is similar to $\eta = \infty$ when the boundary is first approached (i.e.~it trying to leave the boundary) indicates that in this example, it might be better to choose a smaller value for $\eta$. In general, it is hard to decide how to choose a good $\eta$ as it heavily depends on the shape of the boundaries given by the ICs.

Note that in the convergence theory of the descent method in Section~\ref{sec:Descent}, we have only focused on convergence and have not discussed efficiency. Therefore, there are a few things to note when trying to implement this method efficiently.
\begin{remark}
	\begin{enumerate}
		\item If Step 5 in Algorithm \ref{algo:S2} is active, two subproblems will be solved in one iteration. A case where this is obviously very inefficient is when the sequence converges to a point $\bar{x}$ that is nowhere near a boundary. This can be avoided by only solving \eqref{SP2} if $I_\epsilon(x_l) \neq \emptyset$ and only solving \eqref{SP1} if $I_\epsilon(x_l) = \emptyset$ or $\alpha_2(x_l) > -\eta$. In order to avoid solving two subproblems in one iteration entirely, one could consider using the optimal value of the subproblem solved in the previous iteration of the algorithm as an indicator to decide which subproblem to solve in the current iteration. But one has to keep in mind that by doing so, it is slightly more difficult to globalize this method since it does no longer exclusively depend on the current point.
		\item For each evaluation of the projection $\pi$ onto the manifold given by the ECs and active ICs, the problem
		\begin{equation*}
			\begin{aligned}
				& && \underset{y \in \R^n}{\text{min}} &&  \| x-y \|^2,  \\
				& && \text{s.t.} &&  H(y) = 0, \\
				& && && G_i(y) = 0 \quad \forall i \in I,
			\end{aligned}
		\end{equation*}
		for some $I \subseteq \{1,...,m_G\}$ needs to be solved. This is an $n$-dimensional optimization problem with a quadratic objective function and nonlinear equality constraints. The projection is performed multiple times in Steps 9 and 11 and once in Step 13 of Algorithm~\ref{algo:S2} and thus has a large impact on the computational effort. In the convergence theory of Algorithm \ref{algo:S1} and \ref{algo:S2}, we have only used the fact that $R_\pi$ (cf.~Definition \ref{def:projection}) is a retraction and not the explicit definition of $\pi$. This means that we can exchange $R_\pi$ by any other retraction and obtain the same convergence results. By choosing a retraction which is faster to evaluate, we have a chance to greatly improve the efficiency of our algorithm. An example for such a retraction (for $m_H = 1$) is the map $\psi: T \mathcal{M} \rightarrow \mathcal{M}$ which maps $(x,tv)$ with $x \in \mathcal{M}$, $t \in \R$ and $v \in T_x \mathcal{M}$ to 
		\begin{equation*}
			x+tv+s \nabla H(x).
		\end{equation*}
		Here,  $s \in \R$ is the smallest root (by absolute value) of
		\begin{equation*}
			\R \rightarrow \R, s \mapsto H(x+tv+s \nabla H(x)).
		\end{equation*}
		In general, such a map is easier to evaluate than $\pi$. Figure \ref{fig:modProjection} shows the behavior of $\psi$ for $H(x) := x_1^2 + x_2^2 - 1$.
		\begin{figure}[t]
			\centering
	  		\includegraphics[scale=0.55]{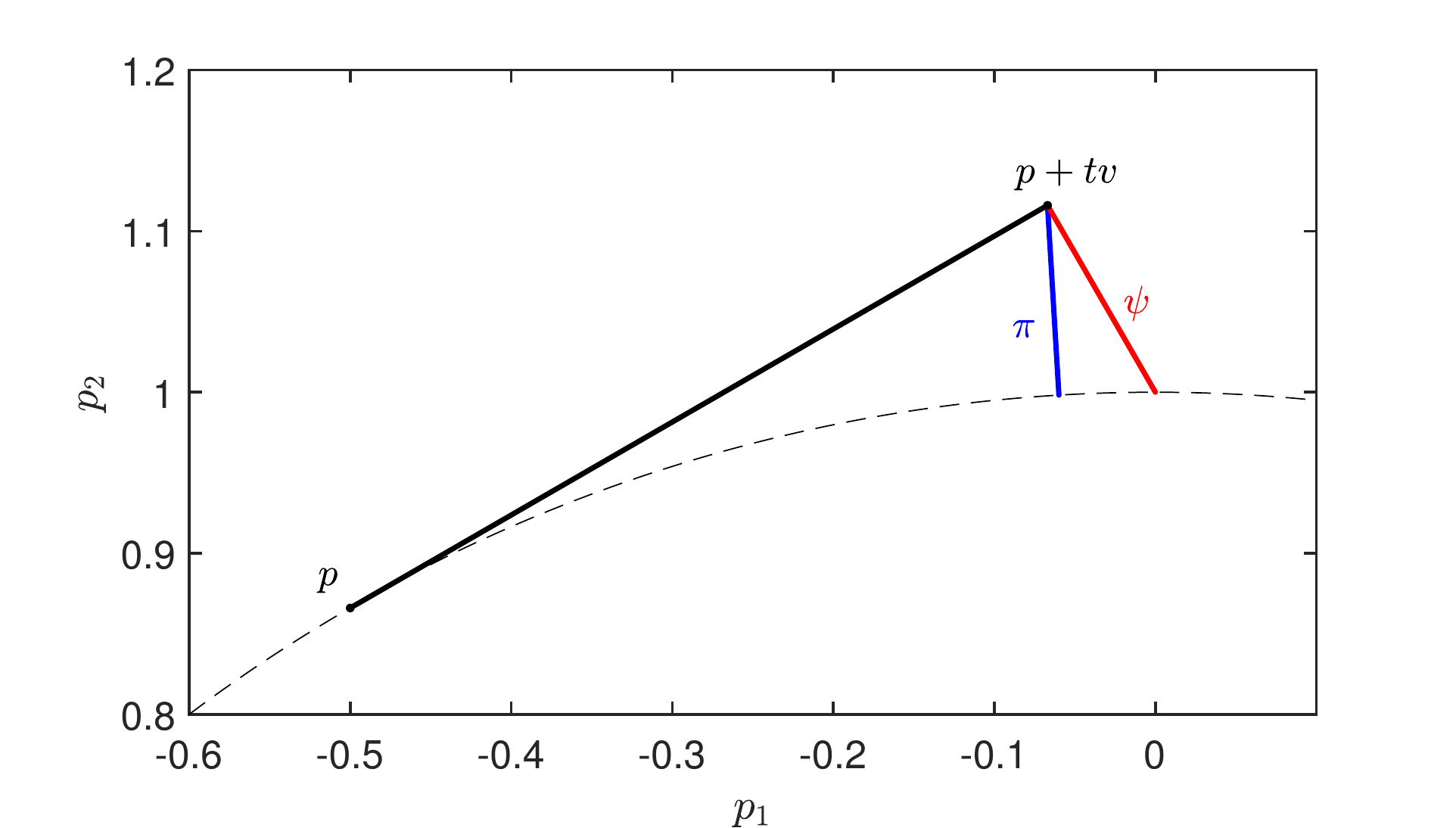}
			\caption{$\psi$ (red) and the projection $\pi$ (blue) on the boundary of the unit circle.}
			\label{fig:modProjection}
		\end{figure}
		For smaller values of $t \| v \|$, these two maps differ even less. If one can show that this map is indeed a retraction (and could possibly be generalized for $m_H > 1$), this would be a good alternative to $\pi$.
	\end{enumerate}
\end{remark}

\section{Conclusion and future work}
\label{sec:Conclusion}
\subsection{Conclusion}
In this article we propose a descent method for equality and inequality constrained MOPs that is based on the steepest descent direction for unconstrained MOPs by Fliege and Svaiter \cite{FS2000}. We begin by incorporating equality constraints using an approach similar to the steepest descent direction on general Riemannian manifolds by Bento, Ferreira and Oliveira \cite{BFO2012} and show Pareto criticality of accumulation points. We then treat inequalities using two different active set strategies. The first one is based on \cite{FS2000} and treats active inequalities as additional objective functions. The second one treats active inequalities as additional equality constraints. Since for the second strategy we require a mechanism to deactivate active inequalities when necessary, we merge both strategies to one algorithm (Algorithm \ref{algo:S2}) and introduce a parameter $\eta$ to control how both strategies interact. We show convergence in the sense that accumulation points of sequences generated by that algorithm satisfy a necessary optimality condition for Pareto optimality. Finally, the typical behavior of our method is shown using an academic example and some of its numerical aspects are discussed, in particular the choice of the parameter $\eta$. In combination with set-oriented methods or evolutionary algorithms, this approach can be used to significantly accelerate the computation of global Pareto sets of constrained MOPs.

\subsection{Future work}
There are several possible future directions which could help to improve both theoretical results as well as the numerical performance.
\begin{itemize}
	\item For the unconstrained steepest descent method in \cite{FS2000}, Fukuda and Drummond present a stronger convergence result than Pareto criticality of accumulation points in \cite{FD2014}. It might be interesting to investigate possible similar extensions for our method.
	\item Being generalizations of the steepest descent method for scalar-valued problems, the unconstrained steepest descent method for MOPs as well as the constrained descent method presented here only use information about the first-order derivatives. In \cite{FDS2009}, a generalization of Newton's method to MOPs is proposed which also uses second-order derivatives. Extending our active set approach to those Newton-like descent methods could help to significantly accelerate the convergence rate.
	\item As explained at the end of Section 4, the computation of the projection onto the manifold can be relatively expensive. The proposed alternative map or similar maps may significantly reduce the computational effort. In order to use such maps without losing the convergence theory, the retraction properties (cf. Definition \ref{def:projection}) need to be shown. 
	\item Since the method presented in this article only calculates single Pareto critical points, globalization strategies and nondominance tests have to be applied to compute the global Pareto set. Since the dynamical system which stems from our descent direction is discontinuous on the boundary of the feasible set, this is not trivial. Although the subdivision algorithm in \cite{DH1997} can be used for this (with minor modifications), it may be possible to develop more efficient techniques that specifically take the discontinuities into account.
	\item Evolutionary algorithms are a popular approach to solve MOPs. It would be interesting to see if our method can be used in a hybridization approach similar to how the unconstrained steepest descent method was used in \cite{LSC2013}.
\end{itemize}

\bibliographystyle{alphaabbr}
\bibliography{literatur}

\newcommand{\etalchar}[1]{$^{#1}$}
\begin{thebibliography}{SWOBD13}

\bibitem[AG90]{AG1990}
E.~L. Allgower and K.~Georg.
\newblock {\em Numerical {C}ontinuation {M}ethods: {A}n {I}ntroduction}.
\newblock Springer-Verlag Berlin Heidelberg, 1990.

\bibitem[AM12]{AM2012}
P.-A. Absil and J.~Malick.
\newblock Projection-like retractions on matrix manifolds.
\newblock {\em SIAM Journal on Optimization, Society for Industrial and Applied
  Mathematics, 2012, 22 (1), 135-158}, 2012.

\bibitem[BFO12]{BFO2012}
G.~C. Bento, O.~P. Ferreira, and P.~R. Oliveira.
\newblock Unconstrained {S}teepest {D}escent {M}ethod for {M}ulticriteria
  {O}ptimization on {R}iemannian {M}anifolds.
\newblock {\em Journal of Optimization Theory and Applications, July 2012, Vol
  154, Issue, 88-107}, 2012.

\bibitem[CCLvV07]{CLV2007}
C.~Coello~Coello, G.~Lamont, and D.~van Veldhuizen.
\newblock {\em {E}volutionary {A}lgorithms for {S}olving {M}ulti-{O}bjective
  {P}roblems}.
\newblock Springer, 2007.

\bibitem[CLM16]{CLM2016}
G.~A. Carrizo, P.~A. Lotito, and M.~C. Maciel.
\newblock Trust region globalization strategy for the nonconvex unconstrained
  multiobjective optimization problem.
\newblock {\em Mathematical Programming}, 159(1):339--369, 2016.

\bibitem[Deb01]{D2001}
K.~Deb.
\newblock {\em Multi-{O}bjective {O}ptimization {U}sing {E}volutionary
  {A}lgorithms}.
\newblock John Wiley \& Sons, Inc., 2001.

\bibitem[DH97]{DH1997}
M.~Dellnitz and A.~Hohmann.
\newblock A {S}ubdivision {A}lgorithm for the {C}omputation of {U}nstable
  {M}anifolds and {G}lobal {A}ttractors.
\newblock {\em Numerische Mathematik}, 75(3):293--317, 1997.

\bibitem[DI04]{DI2004}
L.~G. Drummond and A.~Iusem.
\newblock A projected gradient method for vector optimization problems.
\newblock {\em Computational Optimization and Applications}, 28(1):5--29, 2004.

\bibitem[DSH05]{DSH2005}
M.~Dellnitz, O.~Sch{\"u}tze, and T.~Hestermeyer.
\newblock Covering {P}areto {S}ets by {M}ultilevel {S}ubdivision {T}echniques.
\newblock {\em Journal of Optimization Theory and Applications},
  124(1):113--136, 2005.

\bibitem[Ehr05]{E2005}
M.~Ehrgott.
\newblock {\em Multicriteria {O}ptimization}.
\newblock Springer-Verlag Berlin Heidelberg, 2005.

\bibitem[Eic08]{E2008}
G.~Eichfelder.
\newblock {\em Adaptive {S}calarization {M}ethods in {M}ultiobjective
  {O}ptimization}.
\newblock Springer Berlin Heidelberg, 2008.

\bibitem[FD14]{FD2014}
E.~H. Fukuda and L.~M. G.~A. Drummond.
\newblock A survey on multiobjective descent methods.
\newblock {\em {Pesquisa Operacional}}, 34:585 -- 620, 12 2014.

\bibitem[FDS09]{FDS2009}
J.~Fliege, L.~M.~G. Drummond, and B.~F. Svaiter.
\newblock Newton's {M}ethod for {M}ultiobjective {O}ptimization.
\newblock {\em SIAM Journal on Optimization}, 20(2):602--626, 2009.

\bibitem[FS00]{FS2000}
J.~Fliege and B.~F. Svaiter.
\newblock Steepest {D}escent {M}ethods for {M}ulticriteria {O}ptimization.
\newblock {\em Mathematical Methods of Operations Research. Vol 51 (2000), No
  3, 479-494}, 2000.

\bibitem[FV16]{FV2016}
J.~Fliege and A.~I.~F. Vaz.
\newblock A method for constrained multiobjective optimization based on sqp
  techniques.
\newblock {\em SIAM Journal on Optimization}, 26(4):2091--2119, 2016.

\bibitem[Hil03]{H2003}
S.~Hildebrandt.
\newblock {\em Analysis 2}.
\newblock Springer, 2003.

\bibitem[KT51]{KT51}
H.~W. Kuhn and A.~W. Tucker.
\newblock {Nonlinear programming}.
\newblock In {\em Proceedings of the 2nd Berkeley Symposium on Mathematical and
  Statsitical Probability}, pages 481--492. University of California Press,
  1951.

\bibitem[LSCC13]{LSC2013}
A.~Lara, O.~Sch{\"u}tze, and C.~A. Coello~Coello.
\newblock {\em On Gradient-Based Local Search to Hybridize Multi-objective
  Evolutionary Algorithms}, pages 305--332.
\newblock Springer Berlin Heidelberg, 2013.

\bibitem[Mie98]{M1998}
K.~Miettinen.
\newblock {\em Nonlinear {M}ultiobjective {O}ptimization}.
\newblock Springer US, 1998.

\bibitem[MXX{\etalchar{+}}14]{MXXWM2014}
H.~Mo, Z.~Xu, L.~Xu, Z.~Wu, and H.~Ma.
\newblock Constrained {M}ultiobjective {B}iogeography {O}ptimization
  {A}lgorithm.
\newblock {\em The Scientific World Journal}, 2014.

\bibitem[NW06]{NW2006}
J.~Nocedal and S.~Wright.
\newblock {\em {N}umerical {O}ptimization}.
\newblock Springer, 2006.

\bibitem[RS13]{RS2013}
G.~Rudolph and M.~Schmidt.
\newblock {\em Differential {G}eometry and {M}athematical {P}hysics {P}art 1}.
\newblock Springer, 2013.

\bibitem[SDD05]{SDD2005}
O.~Sch{\"u}tze, A.~Dell'Aere, and M.~Dellnitz.
\newblock On {C}ontinuation {M}ethods for the {N}umerical {T}reatment of
  {M}ulti-{O}bjective {O}ptimization {P}roblems.
\newblock In {\em Practical Approaches to Multi-Objective Optimization}, number
  04461 in Dagstuhl Seminar Proceedings. Internationales Begegnungs- und
  Forschungszentrum f{\"u}r Informatik (IBFI), Schloss Dagstuhl, Germany, 2005.

\bibitem[SMDT03]{SMDT03}
O.~Sch{\"{u}}tze, S.~Mostaghim, M.~Dellnitz, and J.~Teich.
\newblock {Covering Pareto Sets by Multilevel Evolutionary Subdivision
  Techniques}.
\newblock In {\em International Conference on Evolutionary Multi-Criterion
  Optimization (EMO)}, pages 118--132, 2003.

\bibitem[SWOBD13]{SWO+13}
O.~Sch{\"{u}}tze, K.~Witting, S.~Ober-Bl{\"{o}}baum, and M.~Dellnitz.
\newblock {Set Oriented Methods for the Numerical Treatment of Multiobjective
  Optimization Problems}.
\newblock In E.~Tantar, A.-A. Tantar, P.~Bouvry, P.~{Del Moral}, P.~Legrand,
  C.~A. {Coello Coello}, and O.~Sch{\"{u}}tze, editors, {\em EVOLVE - A Bridge
  between Probability, Set Oriented Numerics and Evolutionary Computation},
  volume 447 of {\em Studies in Computational Intelligence}, pages 187--219.
  Springer Berlin Heidelberg, 2013.

\end{thebibliography}

\end{document}